\documentclass[a4paper,10pt]{article}

\usepackage{hyperref}
\usepackage{float}
\usepackage{amssymb}
\usepackage{amsmath}
\usepackage{latexsym}
\usepackage{amsthm}
\usepackage{mathrsfs} 
\usepackage{fancyhdr}
\usepackage{graphicx}
\usepackage{newlfont}
\usepackage{textcomp}
\usepackage{caption}
\usepackage{booktabs}
\usepackage{multirow}
\usepackage{color} 
\usepackage{lipsum}
\usepackage{amsfonts}
\usepackage{epstopdf}
\usepackage{algorithm}
\usepackage{algorithmic}
\usepackage{mathtools}

\usepackage{geometry}
\newtheorem{remark}{Remark}

\newtheorem{proposition}{Proposition}
\newtheorem{corollary}{Corollary}

\newcommand{\sara}[1]{\textcolor{black}{#1}}

\newcommand*{\de}{\mathop{}\!\mathrm{d}}

\DeclarePairedDelimiter{\norm}{\lVert}{\rVert}
\DeclarePairedDelimiter{\abs}{\lvert}{\rvert}

\title{\bf{Exponential Time Differencing for the Tracer Equations Appearing in Primitive Equation Ocean Models}}
\author{%
  S. Calandrini\footnote{Department of Scientific Computing, Florida State University, Tallahassee FL 32306, USA.}%
  \and K. Pieper\footnote{Computer Science and Mathematics Division, Oak Ridge National Laboratory, One Bethel Valley Road, P.O. Box 2008, 
  MS-6211, Oak Ridge, TN 37831, USA.}%
  \and M. D. Gunzburger$^{*}$ 
  }
\date{}

\begin{document}

\maketitle

\begin{abstract}
The tracer equations are part of the primitive equations used in ocean modeling and 
describe the transport of tracers, such as temperature, salinity or chemicals, in the ocean. 
Depending on the number of tracers considered, several equations may be added to and 
coupled to the dynamics system. In many relevant situations, the time-step requirements of
explicit methods imposed by the transport and mixing in the vertical direction are more restrictive
than those for the horizontal, and this may cause the need to use very small time steps if a fully explicit 
method is employed. To overcome this issue, we propose an exponential time 
differencing (ETD) solver where the vertical terms (transport and diffusion) are treated with a matrix exponential, whereas the 
horizontal terms are dealt with in an explicit way. We investigate numerically the computational speed-ups
that can be obtained over other semi-implicit methods, and we analyze the advantages of the method in the case of multiple tracers. 
\end{abstract}



\section{Introduction}

The primitive equations are the constitutive system of equations in ocean modeling. This system is composed of a momentum
equation, an equation for the ocean thickness, equations for the transport of tracers, as well as an equation of state. 
The momentum and thickness equations describe the dynamics, i.e. the change in time of the velocity and thickness of the water.  
The tracer equation describes the transport of tracers, such as temperature, salinity or chemicals. 
Depending on the number of tracers considered, several equations may be added to and coupled to the dynamics system.
The coupling between the tracers and the dynamics depends on the nature of tracers. Temperature and salinity 
impact the density of the layers of ocean water, influencing the dynamics, and for this reason they are called 
active tracers. Hence, between the dynamics and active tracers there is a two-way coupling. 
However, for most tracers the coupling is only one-way, with such tracers being called passive \cite{siberlin2011oceanic}. 

To this day, the numerical solution of the primitive equations remains a challenging task. One of the main issues 
is the presence of multiple time-scales, where different processes (e.g., external and internal gravity waves, 
eddies, biochemical reactions) take different times to be completed. Since the
primitive equations arise from hyperbolic conservation laws (\cite{chiodaroli2017existence, oliger1978theoretical}),
explicit time integrators and Runge-Kutta schemes would hypothetically be good choices for
solving it. However, these schemes are not capable of efficiently handling multiple time-scales, 
because the time-steps restrictions are too severe, resulting in a significant degradation of performance. 
To better handle multiple time-scales, methods have been developed with the help of mathematical 
and algorithmic techniques such as splitting strategies and semi-implicit approaches.
Due to the different properties of the dynamics and the tracer equations, schemes have been created to 
separately deal with the two subsystems (\cite{ringler2013multi, smith2010parallel}). 
Methods for the dynamics need to take into account the different 
wave speeds of the model, with the goal of having a model for which time step sizes are governed by the slow wave 
speeds or the speed of advection and not by the fast wave speeds as is the case for standard explicit schemes.
Examples of time-stepping schemes for the dynamics include implicit (\cite{dukowicz1994implicit}) and 
split-explicit (\cite{higdon2005two, ringler2013multi}) methods.
More recently, exponential time differencing (ETD) methods, also known as exponential integrators, have gained 
attention in the ocean modeling community due to their stability properties that allow time steps considerably 
larger than those dictated by the CFL condition. In \cite{pieper2019exponential}, an ETD scheme has been developed 
for the rotating shallow water equations with multiple horizontal layers, which correspond to a vertical discretization 
of the primitive equations dynamics in an isopycnal vertical coordinate system. 
The main idea behind exponential integrators is a splitting of the right-hand side term of an equation into a linear part 
and a remainder, i.e. \begin{equation*} \partial_t \theta = F(\theta) = A \theta + R(\theta)\;, \end{equation*}
with an appropriate choice of the linear operator $A$. For a review of exponential integrators we refer to 
\cite{hochbruck2010exponential}.

In this work, we devise an ETD method for the tracer equation. 
A tracer is supposed to satisfy a conventional advection-diffusion equation of the form, \begin{equation}
 \label{conv_diff}
 \partial_t \theta + \nabla_{\mathbf{x}}\cdot(\mathbf{u} \theta) + \mathcal{D}_{\mathbf{x}} \theta = q(\theta)
\end{equation}
where $\theta$ is the tracer studied, $\mathbf{u} = (u, w) \in R^3$ is the velocity of water,
which is usually split in to the horizontal velocity \(u \in R^2\) and the vertical
velocity \(w\), $\mathcal{D}$ is a diffusion term, and $q(\theta)$ represents interior sources or
sinks.
The equation~\eqref{conv_diff} has to be solved on a three-dimensional domain, split into
a two-dimensional horizontal and a vertical coordinate \(\mathbf{x} = (x,z)\), and
\(\nabla_{\mathbf{x}} = (\nabla, \partial_z)\). We note that, the tracer equation can usually also be written as
\begin{equation}
 \label{conv_diff_split}
 \partial_t \theta + \nabla_x \cdot(u \theta) + \mathcal{D}_x \theta + \partial_z(w \theta)
 - \partial_z (\kappa_z \partial_z \theta) = q(\theta),
\end{equation}
with the horizontal diffusion \(\mathcal{D}_x\), and the vertical diffusion coefficient
\(\kappa_z\). 
To avoid problems with shocks, the velocity field $\mathbf{u}$ is usually chosen to be divergence free and tangential to the boundary. 

When dealing with the tracer equation, a key point is appropriately including vertical mixing. 
Tracer vertical mixing usually occurs on small time-scales and can be induced by density differences and/or by turbulent motions. 
For explicit time stepping schemes, the time step requirement is usually set by the horizontal advective CFL condition, 
hence very small time steps may need to be used with explicit time stepping methods to include realistic vertical mixing of tracers.  
To avoid this issue, in the time-stepping schemes used by popular ocean models, the vertical diffusion term 
$\partial_z (\kappa_z \partial_z \theta)$ is treated implicitly. 
In POP \cite{smith2010parallel}, the vertical tracer diffusion term is treated with an implicit Euler algorithm, 
whereas the remaining terms of the equation are treated with a leapfrog algorithm. 
In MPAS-Ocean \cite{ringler2013multi}, tracer equations are stepped forward with the mid-time velocity values and 
this process is repeated in a predictor-corrector way. Implicit vertical mixing of tracers completes each time-step, where, 
as in POP, the vertical tracer diffusion term is treated with an implicit Euler algorithm. 

Another important phenomenon to deal with, besides vertical mixing, is when an inflow of cold water 
near the coast leads to cold water on top, scenario that creates density and pressure variations and downward motion. 
In an isopycnal configuration, there is no vertical transport, however, if a mostly Eulerian coordinate 
system is employed (z-star, z-level), rapid variations in the pressure/density may induce a fast vertical flow. 
This phenomenon is amplified by the usage of fine meshes in the vertical, with much smaller vertical than horizontal spacing,  
for example \cite{ringler2013multi} employs \(1\)--\(15\)km horizontal, \(10\)--\(250\)m vertical. 
As we said, many ocean models treat the vertical advection explicitly and so, when a fast transport of water among different 
layers occurs, the model may be unable to appropriately capture this behavior. Consequently, instabilities in the simulation 
may occur, causing the need to decrease the time-step. 
To resolve this issue, we propose an ETD solver where all the vertical terms, i.e.\ vertical advection and diffusion, 
are treated with a matrix exponential, whereas the horizontal are dealt with in an explicit way. 
This means that we are splitting the linear operator $A$ into two parts: $A^z$ that accounts for all vertical terms, and $A^x$  
that contains all horizontal terms. $A^x$ is then incorporated in the remainder $R(\theta)$, so the actual linear 
operator we are working with is $A^z$. This operator splitting has two advantages. 
First, by treating exponentially terms related to fast time-scales, bigger time steps can be taken, and so 
computational speed-ups that can be obtained over other explicit methods. Compared to
semi-implicit methods, we expect higher accuracy due to an exact treatment of the fast scales.
Second, by including only the operator $A^z$ in the exponential, not the whole operator
$A$, the computational cost is reduced, effectively lowering the total computational cost and time of the whole method. 

Another important challenge to face when dealing with tracer equations, and in general with primitive equations, is in 
the numbers of tracers considered. The amount of tracers in an ocean simulation is usually around 40, but may increase 
up to 70, causing a significant computational load. Hence, efficiently solving multiple tracers equations is an important task  
in an ocean model. When an ETD scheme is used, efficiency corresponds to a low-cost evaluation of $\varphi_k$-functions. 
For a given matrix $B$, assembling $\varphi_k(B)$ is generally prohibitive in the large scale context; iterative methods
are used instead, i.e. Krylov subspace algorithms (\cite{eiermann2006restarted, saad1992analysis}). 
\sara{
Scaling and squaring is used usually for dense matrices (\cite{al2009new, dieci2000pade, higham2005scaling,  moler2003nineteen, najfeld1995derivatives}) 
However, it has also been used in the context of multiwavelet Galerkin methods, where a certain level of sparsity can be maintained throughout the  iterations of the method and the approximation of $\varphi_k(B)$ (\cite{archibald2011multiwavelet}). 
In this work, together with the usage of a Krylov method, we pursue a second approach based on scaling and squaring relations, 
focusing on the fact that there is no communication in ocean due to the vertical exponential, since the domain-decomposition 
is horizontal, and the exponential is vertical. The proposed approximation scheme is based on polynomials of moderate degree that result 
in a consistent and stable approximation of $\varphi_k(B)$ with low bandwidth. 
To efficiently solve multiple tracer equations, we intend to preassemble $\varphi_k(B)$ instead of computing $\varphi_k(B)f$ for every right 
hand side $f$. Additionally, since the current ocean models use only 40-100 layers, even a full storage of exponential is feasible.}

Finally, our solver is compared with existing ocean models, to make sure that it is able to reproduce similar results under the same physical 
conditions. 
To do so, the whole primitive equations are solved and two tests from \cite{ilicak2012spurious} are performed. 
The tracer equations is coupled with the dynamics, and to solve this latter system a second ETD solver (Exponential Rosenbrock Euler) is used. 
The results obtained with our ETD solvers are compared with those obtained with three other codes: MPAS-Ocean, MITgcm and MOM.

The paper is organized as follows. Section~\ref{tracer_eq_section} describes the discretization of the tracer equation both in the 
vertical and in the horizontal. Section~\ref{ETD} focuses on exponential integrators and their implementation. 
For the computation of $\varphi_k$-functions, a restarted Krylov subspace method is described, and a scheme based on scaling 
and squaring relations is proposed. For the latter scheme, an error analysis is provided. 
In section~\ref{solver}, the proposed ETD solver for the tracer equation is presented, and the properties of the operator splitting are described. 
The numerical tests are shown and discussed in section~\ref{tests}, while the conclusions follow in section~\ref{sec:conclusions}. 


\section{Discretization of the Tracer Equation} \label{tracer_eq_section}

In this section, the discretization of the tracer equation \eqref{conv_diff} in the vertical and in the horizontal is described. 
In most ocean models, an hydrostatic condition is assumed, leading to describing the primitive equation by the incompressible 
Boussinesq equations in hydrostatic balance. Following this assumption, the tracer equation in continuous form can be re-written as 
\begin{equation}
\dfrac{\partial (\widetilde{\rho} T)}{\partial t} = - \nabla \cdot (\widetilde{\rho} u T) - 
\dfrac{\partial (\widetilde{\rho} T w)}{\partial z} + D^{T}_x + D^{T}_{z} + \mathcal{F^T}\;, \label{tracer} 
\end{equation}
where $T$ is the tracer, $u$ is the horizontal velocity, $w$ is the vertical velocity and $\widetilde{\rho}$ is the pseudo-density.  
The variable $z$ represents the vertical coordinate and it defined positive upward.
$D^{T}_x$ and $D^{T}_{z}$ indicate the horizontal and vertical diffusion term, respectively. 
These terms are defined as 
\begin{align}
D^{T}_x & = \nabla \cdot (h\; \kappa_{x} \nabla T) \;, \\
D^{T}_{z} & = h \frac{\partial}{\partial z} \big(\kappa_{z} \frac{\partial T}{\partial z} \big) \;,
\end{align}
where $\kappa_x$ and $\kappa_{z}$ are the horizontal and vertical diffusion, respectively.
Without loss of generality, we assume that no forcing term is present, i.e. $\mathcal{F^T} = 0$. 
This assumption is equivalent to consider that no external factors have an influence on the tracer behavior.

To discretize equation \eqref{tracer}, we employ z-level coordinates in the vertical \cite{petersen2015evaluation}
and a finite-volume method using a C-grid staggering in the horizontal \cite{ringler2011exploring}. 
For more details about the discretization of all primitive equations, please refer to \cite{ringler2013multi}. 

\subsection{Tracer Equation with vertical discretization}

As in MPAS-Ocean \cite{ringler2013multi, petersen2015evaluation}, the vertical coordinate we use is Arbitrary Lagrangian-Eulerian (ALE). 
With ALE, several coordinate systems can be specified depending on the application. Common choices for the vertical coordinates include 
z-level, where all layers have a fixed thickness except for the top layer, z-star, where all layer thicknesses vary in proportion to 
the sea surface height, and isopycnal, where there is no vertical transport between layers. 

The tracer equation with vertical discretization is
\begin{align}
& \dfrac{\partial (h_{k} T_k)}{\partial t} = - \nabla \cdot (h_k u_k T_k) - 
\overline{T_k} w_k + \overline{T_{k+1}} w_{k+1} + [D^{T}_x]_k + [D^{T}_z]_k \;, \label{tracer_vert_discr}
\end{align}
where $k$ indicates the vertical layer, $k = 1$ is the top layer and $k$ increases downward up to $N$; 
$z = 0$ is the mean elevation of the free surface, and the $z$ coordinate is positive
upward. 
The variable $w_k$ indicates the transport of fluid from layer $k$ to $k-1$, i.e. across the top interface of layer $k$. 
The pseudo-density $\widetilde{\rho}$ has been replaced by the $h$, which is the layer-thickness. 
The operator $\overline{(\cdot)}$, on a generic variable $\psi_k$, is the vertical average between the layer $k$ and the 
above layer $k-1$, i.e.,   
\begin{equation}
\overline{\psi_k} = \dfrac{\psi_{k-1} + \psi_k}{2}\;.
\end{equation}
Finally, $[D^{T}_x]_k$ and $[D^{T}_{z}]_k$ indicate the discretized horizontal and vertical tracer diffusion terms, respectively, 
and are defined as
\begin{equation} [D^{T}_x]_k = \nabla \cdot (h_k \kappa_{x} \nabla T_k)\;, \quad \qquad 
[D^{T}_{z}]_k = h_k \underline{\delta z_k}(\kappa_{z} \overline{\delta z_k}(T_k))\;. \end{equation}
The discrete operators $\overline{\delta z_k}(\cdot)$ and $\underline{\delta z_k}(\cdot)$, on a generic variable $\psi_k$,  
are defined as
\begin{align}
 \overline{\delta z_k} (\psi_k) = \dfrac{\psi_{k-1} - \psi_k}{{\overline{h_k}}}\;, \\
 \underline{\delta z_k} (\psi_k) = \dfrac{\psi_{k} - \psi_{k+1}}{h_k}\;.
\end{align}
The vertical tracer diffusion term can actually be rewritten without introducing the operator $\underline{\delta z_k}$ as 
$[D^{T}_{\nu}]_k = \kappa_{z} (\overline{\delta z_k}(T_k) - \overline{\delta z_k}(T_{k+1}))$.

The choice of the vertical coordinate system is enforced in the computation of the vertical transport. 
$w_k$ can be found by solving the thickness equation for $w_k$. 
The thickness equation discretized in the vertical has the form
\begin{equation} 
 \dfrac{\partial h_k}{\partial t} + \nabla \cdot (h_k {\mathbf u}_k) + w_k - w_{k+1} = 0\;. \label{thickness}
\end{equation} 
For a Boussinesq fluid, this equation represents the continuity equation for the pseudo-density $\widetilde{\rho}$.  
To obtain $w_k$ from \eqref{thickness}, all variables at the previous time step must be known, in particular the time 
derivative of $h$ at layer $k$, $\frac{\partial h_k}{\partial t}$, must be known. 
For this purpose, a new quantity, named $h_k^{ALE}$, is introduced. $h_k^{ALE}$ represents the desired thickness for the new time, 
and it is used to compute $\frac{\partial h_k}{\partial t}$ using a first-order finite difference approximation. 
In this way, $w_k$ can be found as
\begin{equation}
w_k = w_{k+1} - \nabla \cdot (h_k {\mathbf u}_k) - \dfrac{h^{ALE}_k - h_k}{\Delta t}\;. \label{vertical_transport}
\end{equation}
For isopycnal simulations, the vertical transport is set to zero in \eqref{vertical_transport}. For other coordinate systems, like z-level and 
z-star, the way $h_k^{ALE}$ is computed determines the type of coordinates chosen. For example, for z-level vertical coordinates 
\begin{align*}
h_1^{ALE} & = h_1^{rest} + \zeta \;, \\
h_k^{ALE} & = h_k^{rest}, \mbox{ for } k > 1 \;,
\end{align*}
where $h_k^{rest}$ is the layer thickness when the ocean is at rest, and $\zeta$ is the sea surface height defined as 
$\sum_k h_k - \sum_k h_k^{rest}$. 
For z-level coordinates (and for z-type coordinates in general) the resting thickness is considered constant in each horizontal layer,
but for other coordinate systems, like sigma coordinates, $h_k^{rest}$ varies horizontally in proportion to the column's total depth. 
The simulations presented in section \ref{tests} use z-level vertical coordinates.
Please refer to \cite{petersen2015evaluation} for more details about the computation of $h_k^{ALE}$ for other coordinate systems. 

\subsection{Tracer Equation with horizontal discretization} 

The horizontal discretization is a C-grid, finite-volume method applied to a spherical centroidal
Voronoi tessellation (SCVT) mesh.  Height, tracers, pressure and kinetic energy are defined at centers 
of the convex polygons, and the velocity is located at cell edges. 
Vorticity (curl of velocity) is defined at cell vertices. 
In the following, the subscripts $i$ and $e$ indicate the discretized variables through cell centers and edges, respectively. 
Since we are focusing on the discretization of the tracer equation only, we will not work with variables and operators 
defined at cell vertices.

The tracer equation with horizontal discretization is 
\begin{align}
\dfrac{\partial (h_{k,i} T_{k,i})}{\partial t} = & - [\nabla \cdot \big (\widehat{({h}_{k,: })}_e {\mathbf u}_{k,:} \widehat{({T}_{k,:})}_e \big)]_i 
- \overline{T_{k,i}} w_{k,i} + \overline{T_{k+1,i}} w_{k+1,i} \nonumber \\ 
& + [D^{T}_x]_{k,i} + [D^{T}_{z}]_{k,i} \;, \label{tracer_hor_discr} \\
[D^{T}_x]_{k,i} & = [ \nabla \cdot \big( \widehat{({h}_{k,: })}_e \, \kappa_{x} \, [\nabla T_{k,:}]_e \big) ]_i\;, \quad \qquad 
[D^{T}_{z}]_{k,i} = h_{k,i} \underline{\delta z_k}(\kappa_{z} \overline{\delta z_k}(T_{k,i}))\;. \end{align}
Each variable now has two sub-scripted indices, the first indicating the vertical layer, and the second indicating its position on the 
horizontal grid, namely either $i$ or $e$.
Colons in subscripts may be places as second index to indicate that multiple edges or cell centers are used 
in computing the horizontal operator. 
For a generic variable $\psi_k$, the symbol $\widehat{(\psi_{k,:})}_e$ represents the averaging of the variable from two adjacent centers 
to the corresponding edge.
We would like to point out that the vertical transport through the sea surface and at the bottom surface is zero, 
i.e. $w_{1,i} = 0$ and $w_{N+1,i} = 0$. Moreover, we consider $u_{k,e} = 0$ on all boundary edges.

For a generic vector field ${\mathbf Y}_k$ and variable $\psi_k$, 
the discrete horizontal operators $[\nabla \cdot {\mathbf Y}_{k,:}]_i$ and $[\nabla \psi_{k,:}]_e$ are defined as
\begin{align}
[\nabla \cdot {\mathbf Y}_{k,:}]_i & = \frac{1}{A_i} \sum_{e \in E(i)} n_{e,i} Y_{k,e} l_e\;, \\
[\nabla \psi_{k,:}]_e & = \frac{1}{d_e} \sum_{i \in C(e)} - n_{e,i}\psi_{k,i} \;.
\end{align}
$A_i$ indicates the Voronoi cell area, $d_e$ is the distance between cell centers, $l_e$ is edge length 
and $n_{e,i}$ represents the sign of the vector at edge $e$ with respect to cell $i$. 
The sets $E(i)$ are the edges about cell $i$, and the sets $C(e)$ are the cells neighboring edge $e$. 
Thus, the divergence moves from edges to cell-centered quantity, while the gradient moves from cell centers to edges.

\section{Exponential Time Integration} \label{ETD}
This section describes exponential time differencing methods, that later will be used to solve the tracer equation. 
\sara{ETD methods have already been employed to solve the single layer (\cite{clancy2013use, gaudreault2016efficient, luan2019further}) 
and multi-layer (\cite{pieper2019exponential}) shallow water equations.} 

\subsection{Exponential Integrators}
Let $\partial_t T = F(T)$ be a system of partial differential equations (PDEs), where $T=T(t)$ denotes the vector of the solution variables 
for $t \in [t_n, t_{n+1}]$, and $F(T)$ is the right-hand-side term.
The interval $[t_n, t_{n+1}]$ refers to one time step. 
The main idea behind exponential integrators is a splitting of the right-hand-side term into a linear
part and a remainder, i.e. 
\begin{equation}
\partial_t T = F(T) = A_n T + R(T), \label{splitting1}
\end{equation}
where $A_n$ represents a linear operator, and $R(T):=F(T)-A_n T$ denotes the remainder, 
which in general is nonlinear. Applying the variation of constants formula to equation \eqref{splitting1}, 
the solution at time $t_{n+1} = t_n + \Delta t$, i.e. $T_{n+1}=T(t_{n+1})$, is obtained as
\begin{equation}
 T_{n+1} = \exp(\Delta t A_n) T_n + \int_0^{\Delta t} \exp((\Delta t - \tau )A_n) R(T(t_n + \tau)) d\tau\;. \label{variation_of_consts}
\end{equation} 
At this point, to build a concrete exponential integrator, an approximation of $R(T(t_n + \tau))$ must be considered. 
By substituting $R(T(t_n + \tau))$ with its Taylor expansion truncated at $s$ ($s \in \mathbb{N}$), namely    
\[
R(T(t_n + \tau)) = \sum_{k=1}^{s} \dfrac{\tau^{k-1}}{(k-1)!}
\left.
\dfrac{\de^{k-1} R(v(t_n + \tau))}{\de \tau^{k-1}}
\right\rvert_{\tau=0} \,,
\] 
the solution $T_{n+1}$ can be approximated by 
\begin{equation*}
T_{n+1} \approx \exp(\Delta t A_n) T_n + \sum_{k=1}^s \dfrac{1}{(k-1)!} \Big[ \int_0^{\Delta t} \exp((\Delta t- \tau) A_n)\tau^{k-1} 
d\tau \Big] 
\left.
\dfrac{\de^{k-1} R(T(t_n + \tau))}{\de \tau^{k-1}}
\right\rvert_{\tau=0}.
\end{equation*} 
The above expression can be rewritten as
\begin{equation}
T_{n+1} \approx \exp(\Delta t A_n) T_n + \sum_{k=1}^s \Delta t^k \varphi_k(\Delta t A_n)
\left.
\dfrac{\de^{k-1} R(T(t_n + \tau))}{\de \tau^{k-1}}
\right\rvert_{\tau=0} \,,
\label{exp_int}
\end{equation}
by introducing the, so called, $\varphi$-functions defined by
\begin{align}
\varphi_k(\Delta t A_n) & = \dfrac{1}{\Delta t^k(k-1)!} \int_0^{\Delta t} \exp((\Delta t- \tau) A_n)\tau^{k-1} d \tau \label{phi_def1}\\
& = \dfrac{1}{(k-1)!} \int_0^1 \exp((1-\sigma)\Delta t A_n) \sigma^{k-1} d\sigma, \quad k=1,2,\dots,s\;. \label{phi_def2} 
\end{align}
By performing the change of variable $\Delta t - \tau = (1 - \sigma)\Delta t$, \eqref{phi_def2} can be obtained from \eqref{phi_def1}. 
For $k=0$, we have that $\varphi_0(\Delta t A_n) = \exp(\Delta t A_n)$.

The parameter $s$ in \eqref{exp_int} indicates the number of stages of the method.  
By taking $s=1$, the exponential Euler method is obtained as
\begin{equation}
 T_{n+1} \approx \exp(\Delta t A_n) T_n + \Delta t \varphi_1(\Delta t A_n) R(T_n) = T_n + \Delta t \varphi_1(\Delta t A_n) F(T_n)\;, \label{exp_eul}
\end{equation}
where $\varphi_1(\Delta t A_n) = \int_0^1 \exp((1-\sigma) \Delta t A_n) d\sigma = 
(\Delta t A)^{-1}(\exp(\Delta t A_n) - I)$, with $I$ indicating the identity matrix. 
For a generic linear operator $A_n$, this method is first-order accurate, but if $A_n$ is the Jacobian matrix of the system evaluated at $t_n$, 
namely $A_n = F'(T_n)$, then the method becomes second-order accurate \cite{hochbruck2010exponential, luan2019further}. 
In this case, the scheme \eqref{exp_eul} is called Exponential Rosenbrock Euler. 
By taking $s=2$, a second-order single-step method with two stages is obtained as 
\begin{align}
& T_{n}^{(1^{st}\;stage)} = T_n + \Delta t \varphi_1(\Delta t A_n) F(T_n)\;, \label{stage1} \\
& T_{n+1} = T_{n}^{(1^{st}\;stage)} + \Delta t \varphi_2(\Delta t A_n) (R(T_n^{(1^{st}\;stage)}) - R(T_{n}))\;, 
\end{align}
where the first derivative of $N(T_n)$ is approximated with a first-order finite-difference approximation
\[
\left.\dfrac{\de R(T(t_n + \tau))}{\de \tau}\right\rvert_{\tau=0}
\approx
 \dfrac{R(T_n^{(1^{st}\;stage)}) - R(T_{n})}{\Delta t}\;,
\]
and $\varphi_2(\Delta t A_n) = \int_0^1 \exp((1-\sigma)\Delta t A_n) \sigma \, d \sigma
= (\Delta t A_n)^{-2} (\exp(\Delta t A_n) - \Delta t A_n - I)$. 
This scheme is known as ETD2-RK (\cite{hochbruck2005explicit, hochbruck2010exponential}).

\sara{
ETD2-RK fulfills the stiff order conditions of order two described in \cite{hochbruck2005explicit}, section~5.1.  
However, by relaxing some of these conditions, different schemes can be obtained that are more convenient for computations. 
For instance, the following two-stage predictor-corrector scheme makes use only of the $\varphi_1$ function. 
\begin{align}
& T_{n}^{(1^{st}\;stage)} = T_n + \Delta t \varphi_1(\Delta t A_n) F(T_n)\;, \label{stage1_predictor_corrector} \\
& T_{n+1} = T_{n}^{(1^{st}\;stage)} + \dfrac{1}{2} \Delta t \varphi_1(\Delta t A_n) (R(T_n^{(1^{st}\;stage)}) - R(T_{n}))\;. 
\label{stage2_predictor_corrector}
\end{align}
This scheme fulfills the nonstiff order conditions up to order two, and the stiff order conditions up to order one. 
From an implementational point of view, it is simpler than ETD2-RK, since \(\varphi_2\) does not need to be assembled, 
and a further advantage in terms of computational time can be obtained if the matrix $\varphi_1(\Delta t A_n)$ can be 
precomputed and stored efficiently. 
}

\subsection{Computation of the $\varphi$-functions}
The computation of the $\varphi$-functions is of major relevance for ETD methods. 
Let us assume that $A$ is in $\mathbb{R}^{n \times n}$, where $b \in \mathbb{R}^n$ 
and $k \in \mathbb{N}$ is the index of the $\varphi$-function. 
Fist of all, let us consider the special case $k=0$, for which we have $\varphi_0(A) = \exp(A)$.
The most popular method for computing the matrix exponential $\exp(A)$ is the scaling and squaring algorithm 
\cite{moler1978nineteen, moler2003nineteen}.   
In \cite{higham2005scaling}, the scaling and squaring algorithm used by the MATLAB
function \textit{expm} is described, and a variation of this method that alleviates overscaling problems is presented in \cite{al2009new}. 
The Expokit package \cite{sidje1998expokit} uses a scaling and squaring method as well for computing matrix exponentials.  
An alternative to scaling and squaring methods is given by Krylov subspace projections \cite{hochbruck1997krylov, saad1992analysis}. 
\sara{Krylov schemes can be used for evaluating any $\varphi_k$-function with index $k \ge 0$.} 
With Krylov schemes, matrix-vector products of the form $\varphi_k(A) b$ are computed, without 
the actual construction of the matrix $\varphi_k(A)$. 
To overcome some memory issues associated with standard Krylov methods, restarted schemes have been developed 
\cite{eiermann2006restarted}. The C/C++ library SLEPc \cite{roman2015slepc} uses the restarted algorithm presented in \cite{eiermann2006restarted}.  

In the following, a brief description of the algorithm presented in \cite{eiermann2006restarted} is given, since this method is used in the 
numerical results section, and scaling and squaring algorithm for indexes $k \ge 0$ is presented, for which an error estimate is derived.  

\subsubsection{Krylov subspace approximation}\label{restarted_kry}
The Krylov subspace approximation has been found very efficient to 
compute matrix-vector products $\varphi_k(A) b$ due to the optimality of the matrix polynomials produced by Krylov methods. 
The idea behind a Krylov subspace approach is to approximate the vector $\varphi_k(A)b$, which lives in $\mathbb{R}^n$, 
in a smaller space of dimension $m$. The Krylov approximation of $\varphi_k(A)b$ is based on an Arnoldi decomposition of $A$ 
\begin{equation}
A V_m = V_{m+1} \widetilde{H}_m = V_m H_m + \eta_{m+1,m} v_{m+1} e^T_m \;,
\end{equation}
where $V_m$ is a $m \times n$ matrix whose columns form an orthogonal basis for the Krylov subspace of dimension $m$  
\begin{equation*}
K_m(A, b) = \mbox{span}\{b, Ab, \dots, A^{m-1}b \}\;,
\end{equation*} 
$\widetilde{H}_m = [\eta_{i,j}]$ is an $(m + 1) \times m$ upper Hessenberg matrix, $H_m = [I_m \; 0]\widetilde{H}_m$ and $e_m$ denotes the $m$th
vector of the standard basis of $\mathbb{R}^m$. 
The matrices $V_m$ and $H_m$ are such that \begin{equation} H_m = V_m^T A V_m\;, \label {arnoldi_consequence}\end{equation} 
therefore $H_m$ can be seen as the projection of the action of $A$ to the Krylov subspace $K_m(A, b)$.   

Standard Krylov methods requires the storage of $V_m$, i.e. the storage of 
$m$ vectors of size $n$, which may be costly for moderate to large values of $m$. 
Therefore, a standard Krylov subspace method may be impractical because of the storage requirements associated with $V_m$.
To overcome this issue, the Arnoldi approximation could be modified in a way that allows the construction of successively 
better approximations of $\varphi_k(A)b$ based on a sequence of Krylov spaces of small dimension. 
Methods based on this approach are called restarted Krylov subspace methods, and they use an Arnoldi-like decomposition 
where a sequence of ascending (not necessarily orthonormal) basis vectors is introduced. 
We now focus on the restarted Krylov subspace algorithm presented in \cite{eiermann2006restarted}, 
which can be summarized as follows.
An Arnoldi-like decomposition of $A$ is constructed \begin{equation}
A \widehat{V}_p = \widehat{V}_p \widehat{H}_p + n_{p+1} v_{pm+1} e_{pm}^T \;, \label{arnoldi_like}                
\end{equation}
where $\widehat{V}_p=[V_1 \; V_2 \; \cdots \; V_p] \in \mathbb{R}^{n \times pm}$, 
$\widehat{H}_p = \begin{bmatrix}
H_1 & \\
E_2 & H_2 & \\
& \ddots & \ddots &\\
& & E_p & H_p
\end{bmatrix} \in \mathbb{R}^{pm \times pm}$, and \\
$E_j=\eta_j e_1 e_m^T \in \mathbb{R}^{m \times m}$, $j=2, \dots, p$. 
Since \eqref{arnoldi_like} is an Arnoldi-like decomposition, the columns of $\widehat{V}_k$ are only blockwise orthonormal. 
The matrices $V_1,V_2, \dots, V_p \in \mathbb{R}^{n \times m}$, $H_1,H_2, \dots, H_p \in \mathbb{R}^{m \times m}$ 
and the scalars $\eta_2, \eta_3, \dots, \eta_{p+1}$ 
are obtained from $p$ proper Arnoldi decompositions.
Setting $$\varphi_k(\widehat{H}_p) = \begin{bmatrix}
\varphi_k^{1,1} & \vspace{0.05 cm}\\
\varphi_k^{2,1} & \varphi_k^{2,2} & \\
\vdots & \vdots & \ddots & \\
\varphi_k^{p,1} & \varphi_k^{p,2} & \dots & \varphi_k^{p,p}
\end{bmatrix}, \quad \mbox{where } \; \varphi_k^{j,j} = \varphi_k(H_j), \;\; j=1, 2, \dots, p \;,$$
the approximation $\varphi_k^p$ of $\varphi_k(A)b$ after $p$ restart cycles is given by
\begin{equation}
\varphi_k^p = \widehat{V}_p \varphi_k(\widehat{H}_p)e_1 = [V_1 \; V_2 \; \cdots \; V_p] \varphi_k(\widehat{H}_p)e_1 = 
\sum_{j=1}^p V_j \varphi_k^{j,1} e_1 = \varphi_k^{p-1} + V_p \varphi_k^{p,1} e_1 \;.
\end{equation} 
Therefore, the approximation $\varphi_k^{p}$ is obtained from the previous approximation $\varphi_k^{p-1}$ plus a correction term. 
Only $\varphi_k^{p-1}$ has to be stored from the previous cycle of the algorithm, and the matrix $V_{p-1}$ 
(together with $V_{p-2}, \dots ,V_1$) can be discarded after computing $\varphi_k^{p-1}$.
An efficient implementation of this algorithm can be found in \cite{afanasjew2008implementation}, where it is also shown 
how to stably compute the coefficient vector $\varphi_k^{p,1} e_1$. 

\subsubsection{Scaling and Squaring}\label{scaling}
We now present a scaling and squaring method for the computation of $\varphi_k(A)$. 
We are going to base the scaling and squaring method on the recursive relations for the
\(\varphi_k\)-functions; cf.~\cite{beylkin1998new, koikari2007error, SkaflestadWright:2009}. 
They are given as (see~\cite[Lemma~3]{SkaflestadWright:2009}):
\begin{align}
\label{eq:phi_simple}
&2^k\varphi_k(2A) = \exp(A)\varphi_k(A) + \sum_{j=0}^{k-1}\frac{\varphi_{k-j}(A)}{j!}  \\
\label{eq:phi_BKV}
=&
\begin{cases}
\varphi_{k/2}(A)^2 + 2 \sum_{j=0}^{k/2} \frac{1}{j!} \varphi_{k-j}(A)
&\text{for } k \text{ even,} \\
\varphi_{(k-1)/2}(A)\varphi_{(k+1)/2}(A) + 2 \sum_{j=0}^{(k-1)/2} \frac{\varphi_{k-j}(A)}{j!}
+ \frac{\varphi_{((k+1)/2)}(A)}{((k+1)/2)!}
&\text{for } k \text{ odd.} \\  
\end{cases}
\end{align}
Note that in both expressions the matrix function of the matrix \(2A\) is expressed as a
product of two matrix functions of \(A\) plus some correction terms.
The first identity is simpler, but the second one reduces the number of correction terms,
which is slightly more efficient for computations; however \cite{SkaflestadWright:2009}
mentions that the first form is more stable in numerical experiments. 

The first formula follows from the definition~\eqref{phi_def1} using \(\Delta{t} = 2\),
and a splitting of the integral into the
intervals \((0,1)\) and \((1,2)\)
\begin{align*}
2^k\varphi_k(2A) &= \int_0^1 \exp((2- \tau)A) \frac{\tau^{k-1}}{(k-1)!} \de \tau
                + \int_1^2 \exp((2-\tau)A) \frac{\tau^{k-1}}{(k-1)!} \de \tau \\
&= \exp(A) \int_0^1 \exp((1-\tau)A) \frac{\tau^{k-1}}{(k-1)!}
+ \int_0^1 \exp((1-\tau)A) \frac{(1+\tau)^{k-1}}{(k-1)!} \de \tau,
\end{align*}
where the second integral was shifted to \((0,1)\).
Now, using the binomial identity \((1+\tau)^{k-1}/(k-1)!
 = \sum_{j=0}^{k-1} \tau^{k-j-1} / ((k-1-j)!\,j!)\) in the second term and the definition of
\(\varphi_k\) and \(\varphi_{k-j}\) shows~\eqref{eq:phi_simple}.
Now, to derive the second identity, we use in a first step
\begin{multline*}
\exp(A)\varphi_k(A) = (I + \varphi_1(A)A)\varphi_k(A)
= \varphi_k(A) + \varphi_1(A) (A\varphi_k(A)) \\
= \varphi_k(A) + \varphi_1(A) (\varphi_{k-1}(A) - I/(k-1)!)
= \varphi_1(A) \varphi_{k-1}(A) + \varphi_k(A)  - \varphi_1(A)/(k-1)!.
\end{multline*}
Proceeding iteratively \(m\) times with the first term, we obtain
\[
\exp(A)\varphi_k(A) = \varphi_{m}(A)\varphi_{k-m}(A) +  \sum_{j=1}^m\frac{\varphi_{k-j}(A)}{j!} - \sum_{j=k-m-1}^{k-1} \frac{\varphi_{k-j}(A)}{j!}.
\]
using this identity for \(m = k/2\) and \(m = (k-1)/2\) respectively, we obtain~\eqref{eq:phi_BKV}.

In order to evaluate the matrix functions, we are going exploit these recursive relations
to reduce the computation to matrix functions of a scaled matrix \(A/2^M\).
Then, we use a polynomial approximation for this matrix. In order to ensure accuracy and
stability, we do this in the following way: First, let
\[
p_0^{0}(z) = T_{r}(z) + z^{r+1}q(z) = \exp(z) + \mathcal{O}(\abs{z}^{r+1})
\]
be a polynomial of degree \(N_p\) that approximates the exponential function up to order
\(r\). Here, \(T_r(z) = 1+z + \dotsc + z^r/r!\) is the Taylor approximation to \(\exp\),
and \(q(z)\) is a remainder.
Then, for all \(k \leq r+1\) we define the consistent polynomial approximations to the
\(\varphi_k\)-functions as
\begin{equation}
\label{eq:poly_init}
p_k^{0}(z) = z^{-k}\left(p_0^{M}(z) - T_k(z)\right)
 = \sum_{j=0}^{r-k} \frac{z^{j}}{(j-k)!} + z^{r+1-k}q(z)
\end{equation}
Now, we define the higher order recursive approximations for \(M>0\) to \(\varphi_k\)
using~\eqref{eq:phi_simple} as
\begin{equation}
\label{eq:poly_rec}
p_k^{M}(z) = 2^{-k}\left(p^{M-1}_0(z/2)p^{M-1}_k(z/2)
 + \sum_{j=0}^{k-1} \frac{p_{k-j}^{M-1}(z/2)}{j!} \right).
\end{equation}
This definition ensures that the resulting \(p_k^{r,M}\) functions have similar properties as
the original \(\varphi\) functions.
\begin{proposition}
For any \(M \geq 0\) and \(1\leq k\leq r+1\) there holds
\begin{align}
\label{eq:pk_properties_1}
p_k^{M}(z) &= z^{-1}\left(p_{k-1}^M(z) - 1/(k-1)!\right),\\
\label{eq:pk_properties}
p_k^{M}(z) &= z^{-k}\left(p^M_0(z) - T_{k-1}(z)\right).
\end{align}
\end{proposition}
\begin{proof}
For \(M = 0\) this holds according to definition. For higher \(M\) we follow an induction
argument.

First, we show~\eqref{eq:pk_properties_1}.
We use the recursive definition~\eqref{eq:poly_rec} to obtain
\[
z p_k^M(z) = \frac{1}{2^{k-1}}\left(p^{M-1}_0(z/2) (z/2) p^{M-1}_k(z/2) +
  \sum_{j=0}^{k-1} \frac{(z/2) p^{M-1}_{k-j}(z/2)}{j!}\right)
\]
and use the induction hypothesis to obtain
\begin{align*}
2^{k-1} z p_k^M(z) &= 
 p^{M-1}_0(z/2) \left(p_{k-1}^{M-1}(z/2) - \frac{1}{(k-1)!}\right)
 + \sum_{j=0}^{k-1} \frac{p^{M-1}_{k-j-1}(z/2) - 1/(k-j-1)!}{j!} \\
  &= p^{M-1}_0(z/2) p^{M-1}_{k-1}(z/2) + \sum_{j=0}^{k-2} \frac{p^{M-1}_{k-1-j}(z/2)}{j!}
    - \sum_{j=0}^{k-1}\frac{1}{j!(k-j-1)!} \\
  &= 2^{k-1}p^{M}_{k-1}(z) - 2^{k-1}/(k-1)!,
\end{align*}
using again the recursive definition~\eqref{eq:poly_rec} for \(k-1\) and the well-known
summation formula of binomial coefficients. Dividing by \(2^{k-1}z\) yields~\eqref{eq:pk_properties_1}.
Concerning~\eqref{eq:pk_properties}, we note that is suffices to repeatedly
apply~\eqref{eq:pk_properties_1} for \(k\), \(k-1\), ..., \(1\).
\end{proof}
\begin{remark}
The above result also implies that the recursion from~\eqref{eq:phi_BKV} is equivalent for
the construction of \(p_k^M\), since the equality~\eqref{eq:pk_properties_1} can be used to
convert between both versions.
\end{remark}

\begin{remark} 
In \eqref{stage1_pc_tracer}-\eqref{stage2_pc_tracer}, which is the ETD method we actually use to later solve the tracer equation, 
the only $\varphi_k$-function needed is $\varphi_1$. 
For this specific case, the recursion formula \eqref{eq:phi_simple} looks like 
\begin{equation*}
\varphi_1(A) = \frac{1}{2} \big(\exp \big(\frac{1}{2} A \big) + I\big) \varphi_1(\frac{1}{2} A)\;. 
\end{equation*} 
and \begin{equation*} 
p_1^{M}(z) = z^{-1}(p_0^{M}(z)-1)\;.       
\end{equation*} 
Using the above relations, the following algorithm computes $\varphi_1(A)$ from $\varphi_1(A/2^M)$.
\begin{algorithm}
 \begin{enumerate}
 \item[Step 1.] Define 
 $$\sara{p_1^{0}(A) = \frac{A^{-1}}{2^M}(p_0^{0}(A/2^M) - I)\;,\;\;\;\; p_0^{0}(A/2^M) = \exp(A/2^M) + \mathcal{O}(\abs{A/2^M}^{r+1}) \;.}$$
 \item[Step 2.] For $0 \leq j \leq M$, Given $p_0^{j}(\frac{A}{2^{M+j-1}})$ and $p_1^{j}(\frac{A}{2^{M+j-1}})$, 
 compute $p_1^{j}(\frac{A}{2^{M+j}})$ as 
 $$p_1^{s,j}\Big(\frac{A}{2^{M+j}}\Big) = \frac{1}{2} \Big( p_0^{j}\Big(\frac{A}{2^{M+j-1}}\Big) + I \Big)
 p_1^{j}\big(\frac{A}{2^{M+j-1}}\big)\;.$$
 \item[Step 3.] Given $p_0^{j}(\frac{A}{2^{M+j-1}})$, compute $p_0^{j}(\frac{A}{2^{M+j}})$ as 
 $$p_0^{j}(\frac{A}{2^{M+j}}) = p_0^{j}(\frac{A}{2^{M+j-1}})) \; p_0^{j}(\frac{A}{2^{M+j-1}})\;.$$
 \end{enumerate}
 \caption{Computation of $\varphi_1(A)$} \label{phi1_scaling}
\end{algorithm}
\end{remark}

Finally, we provide an error estimate for the approximation \(\varphi_k \approx p_k^M\).
First, we consider the polynomial approximation on a subset of the complex plane. To prepare for the general
case, we let \(\Sigma \subset \mathbb{C}^{-} + \rho_0 \subset \mathbb{C}\) be some compact subset of the
negative complex plane shifted by \(\rho_0 \geq 0\), and assume that the
underlying polynomial fulfills the stability assumption
\begin{equation}
\label{eq:assu_poly_stab}
\abs{p_0^0(z)} \leq \exp(\tau\rho_0) \quad\text{for all } z \in \tau\Sigma,
\text{ where } 0 \leq \tau \leq 1.
\end{equation}
Moreover, since \(\Sigma\) is compact, there exists \(c_q > 0\) such that
\begin{equation}
\label{eq:poly_approx}
\abs{p_0^0(z) - \exp(z)} \leq c_q \abs{z}^{r+1} \quad\text{for all } z \in \Sigma.
\end{equation}
This will be the basis of the error estimates.
\begin{remark}
There are many situations where this assumption is fulfilled. In the case that
\(p_0^0(z)\) is the Taylor polynomial \(T_r(z)\) and \(\rho_0 = 0\), \(\Sigma\) can be
chosen as the intersection of the negative half-plane \(\mathbb{C}^-\) and the well-known
stability region of a Runge-Kutta scheme of order \(r\).
In this case, due to the
relation \(T_r(z) - \exp(z) = z^{r+1}\varphi_{r+1}(z)\) and \(\abs{\varphi_{r+1}(z)} \leq
\varphi_{r+1}(\operatorname{Re}z) \leq 1/(r+1)!\) we have \(c_q = 1/(r+1)!\). 
For an overview over known results on polynomials with optimal stability properties for
various forms of \(\Sigma\) and a computational approach to determine them, we refer to
Ketcheson~\cite{ketcheson2013optimal}.
\end{remark}
We first investigate the case of the matrix exponential \(k=0\).
\begin{proposition}
\label{prop:err_est_scsq}
Let the polynomial \(p_0^0\) fulfill~\eqref{eq:assu_poly_stab}. Then for all \(M \geq 0\) 
we have the stability
and approximation properties:
\begin{align}
\label{eq:stab_M}
\abs{p_0^M(z)} &\leq \exp(\tau\rho_0), \\
\label{eq:approx_M}
\abs{p_0^M(z) - \exp(z)} &\leq c_q \exp(\tau\rho_0) \abs{z}^{r+1} 2^{-Mr},
\end{align}
which hold for all \(z \in \tau \Sigma\), for the enlarged stability radii \(0 \leq \tau \leq 2^M\).
\end{proposition}
\begin{proof}
The first statement follows in a straightforward way from~\eqref{eq:assu_poly_stab}
since \(p_0^M(z) = p_0^0(z/2^M)^M\). For the second
statement, we use an induction argument, noting that the case \(M=0\) follows
directly from~\eqref{eq:poly_approx}. For \(M > 0\), we let \(z \in \tau\Sigma\) be
arbitrary we use~\eqref{eq:poly_rec} to obtain
\begin{align*}
\abs*{\exp(z) - p_0^M(z)}
&= \abs*{\left(\exp(z/2) + p_0^{M-1}(z/2)\right)\left(\exp(z/2) - p_0^{M-1}(z/2)\right)} \\
&\leq 2 \exp(\tau\rho_0/2)\, c_q\exp(\tau\rho_0/2) \abs*{z/2}^{r+1} 2^{-(M-1)r} \\
&= \exp(\tau\rho_0) \abs*{z}^{r+1} 2^{-Mr},
\end{align*}
where we have used \(z/2 \in (\tau/2)\Sigma\) and the induction hypothesis
together with~\eqref{eq:stab_M} and \(\abs{\exp(z/2)} =
\exp(\operatorname{Re(z)/2}) \leq \exp(\tau\rho_0/2)\) in the second step.
\end{proof}
A similar estimate follows now also for the higher \(\varphi\)-functions.
\begin{corollary}
\label{cor:err_est_scsq}
Let \(p_0^0\) fulfill~\eqref{eq:assu_poly_stab}. Then for \(k \leq r+1\) we have
\[
\abs{p_k^M(z) - \varphi_k(z)} \leq c_q \exp(\tau\rho_0) \abs{z}^{r+1-k} 2^{-Mr}.
\]
for all \(z \in \tau\Sigma\) and \(0\leq \tau \leq 2^M\).
\end{corollary}
\begin{proof}
It suffices to use~\eqref{eq:pk_properties} to write
\[
p_k^M(z) - \varphi_k(z) = z^{-k}(p_0^M - \exp(z)),
\]
where we can apply Proposition~\ref{prop:err_est_scsq}.
\end{proof}

For the case of a diagonalizable matrix, one can now obtain an error
estimate in the standard way.
\begin{corollary}
\label{cor:err_est_matr}
Let \(A = V D V^{-1}\), where \(D\) is a diagonal matrix with the eigenvalues \(\sigma(A)
\subset 2^M \Sigma\) on the diagonal. Then if~\eqref{eq:assu_poly_stab} holds, we have
\[
\norm*{\varphi_k(A) - p_k^M(A)}_2 \leq  c_q \operatorname{cond}(V) \exp(\rho_0)
\abs{A}^{r+1-k} \, 2^{-Mr},
\]
where \(\norm{\cdot}_2\) is the matrix norm induced by the Euclidean norm, \(\abs{A}\) is
the spectral radius of \(A\), and
\(\operatorname{cond}(V) = \norm{V}_2\norm{V^{-1}}_2\) is the condition number of \(V\).
\end{corollary} 


\section{ETD Solver with Operator Splitting for the Tracer Equation} \label{solver}

In this section, the ETD solver developed for the tracer equation is presented. 
We focus on the description of the linear operator chosen, whose structure strictly depends on the physics of the problem, 
i.e. processes occurring at different time-scales. 
In the ocean, most processes occur at large scales, as forcing by the wind or by currents in the upper layers of the ocean, 
but in many relevant situations (e.g., eddies, cold water flowing over warm water and vice versa), the vertical transport and mixing 
of tracers follow much faster time-scales. 
Hence, semi-implicit methods are needed for the appropriate inclusion of vertical transport and mixing in the model using 
time-steps that are not excessively small. 
Using an ETD scheme where the linear operator is split into a vertical and a horizontal part and the vertical part is treated exponentially, 
larger time steps can be taken and still have an appropriate description of the processes happening in the vertical. 
While models like MPAS-Ocean and POP treat only the tracer vertical diffusion implicitly, we treat with a matrix exponential all vertical terms 
of the tracer equation, i.e. vertical advection and diffusion. 

\subsection{Operator Splitting} \label{jac_splitting}
Let us rewrite the tracer equation \eqref{tracer_hor_discr} as equation \eqref{splitting1}, that is
\begin{equation}
\partial_t T = F(T) = J_n T + R(T), \label{tracer_split1}
\end{equation}
where $T=T(t)$ denotes the vector of tracer values for $t \in [t_n,t_{n+1}]$, and the linear operator is the Jacobian 
of the system evaluated at $t_n$.
Since we are assuming a zero forcing term, the tracer equation is linear in $T$, which implies that the nonlinear 
reminder $R(T)$ is zero, and so \eqref{tracer_split1} becomes 
\begin{equation} \partial_t T = F(T) = J_n T\;.\end{equation} 
At this point, we split the Jacobian $J_n$ into a vertical and a horizontal part, i.e.   
\begin{equation} J_n = J_n^z + J_n^x\;, \label{oper_split} \end{equation} 
where $J_n^z$ contains the derivatives of the vertical terms only, and $J_n^x$ contains the derivatives of the horizontal terms only. 
Thus, 
\begin{equation} \partial_t T = F(T) = J_n^z T + J_n^x T\;. \label{tracer_split2} \end{equation}
An exponential integrator can be applied to solve \eqref{tracer_split2}.
Both terms $J_n^z T$ and $J_n^x T$ are linear, so either one of them can be considered as
the linear part of the equation or the remainder. 
Since the vertical processes are those occurring at fast time-scales, the vertical terms have to be treated with a matrix exponential.   
For this reason, the term $J_n^z T$ is interpreted as the linear part, 
while $J_n^x T$ is the remainder, which is linear as well in this case. 
Thus, we construct an exponential time differencing solver where the vertical terms 
are treated implicitly with a matrix exponential, whereas the horizontal are dealt with in an explicit way. 
The scheme has the following form: 
\begin{align}
& T_{n}^{1st\,stage} = T_n + \Delta t \varphi_1(\Delta t J_n^z)F(T_n)\;, \label{stage1_pc_tracer}\\
& T_{n+1} = T_{n}^{1st\,stage} + 
\dfrac{1}{2}\Delta t \varphi_1(\Delta t J_n^z) (R^{1st\,stage}_{n}-R_n)\;, \label{stage2_pc_tracer}
\end{align}
i.e. it is a two-stage ETD method following a predictor-corrector.  
The remainders are defined as $R_n = F(T_n) - J_n^z T_n = J_n^x T_n$ and 
$R^{1st\,stage}_{n} = F(T_{n}^{1st\,stage}) - J_n^z T_{n}^{1st\,stage} = J_n^x T_{n}^{1st\,stage}$.  
Both $R_n$ and $R^{1st\,stage}_n$ take into account only the contributions from the horizontal terms, but 
computationally $J_n^x$ is never built. $R_n$ is obtained for free from the construction of the right-hand side $F(T_n)$, which is needed 
in the first stage, and the construction of $R^{1st\,stage}_n$ is cheap because we do not need to construct the full right-hand side $F(T_n^{1st\,stage})$ but only its horizontal terms evaluated at $T_n^{1st\,stage}$. 

Please note that we could have chosen to treat the full Jacobian $J_n$ with a matrix exponential, 
and this would have given us the exact solution (up to machine precision) of equation \eqref{tracer_hor_discr}, since it is 
linear in $T$. From a computational point of view, this strategy is not appealing since the construction of $J_n$ 
might be very expensive in terms of time and computational cost. 
\sara{Given that the fast vertical scales are those responsible for the most restrictive advective CFL, one could use a domain 
decomposition approach in the horizontal. However, the advantage given by the domain decomposition would be lost if the 
whole Jacobian is treated with a matrix exponential, because all the domains would be coupled.} 
We think that the exponential Euler method would not have been a valid choice either, since it is likely too inaccurate to be useful. 


\subsection{Structure of the matrix $J_n^z$}
Another advantage of the operator splitting \eqref{oper_split} relies in the structure of the matrix $J_n^z$. 
This matrix contains the derivatives of the vertical terms only, and its dimension is 
$N_z$ $\times$ $N_x$,
where $N_z$ indicates the total number of vertical layers, and $N_x$ is the total number 
of elements in the horizontal discretization, i.e. the total number of Voronoi cells.
The entries of $J_n^z$ can be ordered so that the derivatives associated with the same horizontal element form a submatrix of dimension 
$N_z \times N_z$ . This is possible since, for every layer, there is no interaction  
between the derivatives of the vertical terms associated with two different Voronoi cells. 
Therefore, $J_n^z$ has a block diagonal structure, 
\begin{equation}
J_n^z = \begin{bmatrix}
J_n^{z,1} & 0 & \dots & 0\\
0 & J_n^{z,2} & \dots & 0\\
\vdots & \dots & \dots & \vdots\\
0 & 0 & \dots & J_n^{z,N_x}
\end{bmatrix}
\end{equation}
where each block $J_n^{z,i}$ represents the contributions (derivatives of the vertical terms) given by a single element $i$ 
in the horizontal discretization, and the dimension of each block depends on the number of vertical layers. 
Figure \ref{block_diag} shows the structure of $J_n^z$ and each 
diagonal block in a simple case with 4 Voronoi cells and 4 vertical layers. 
Since we are dealing with a one dimensional domain in the horizontal, the diagonal blocks are banded matrices; in particular 
they are tridiagonal matrices because we use up-winding in the \sara{horizontal} discretization. 
\begin{figure}  
\begin{center}
\includegraphics[scale=0.5]{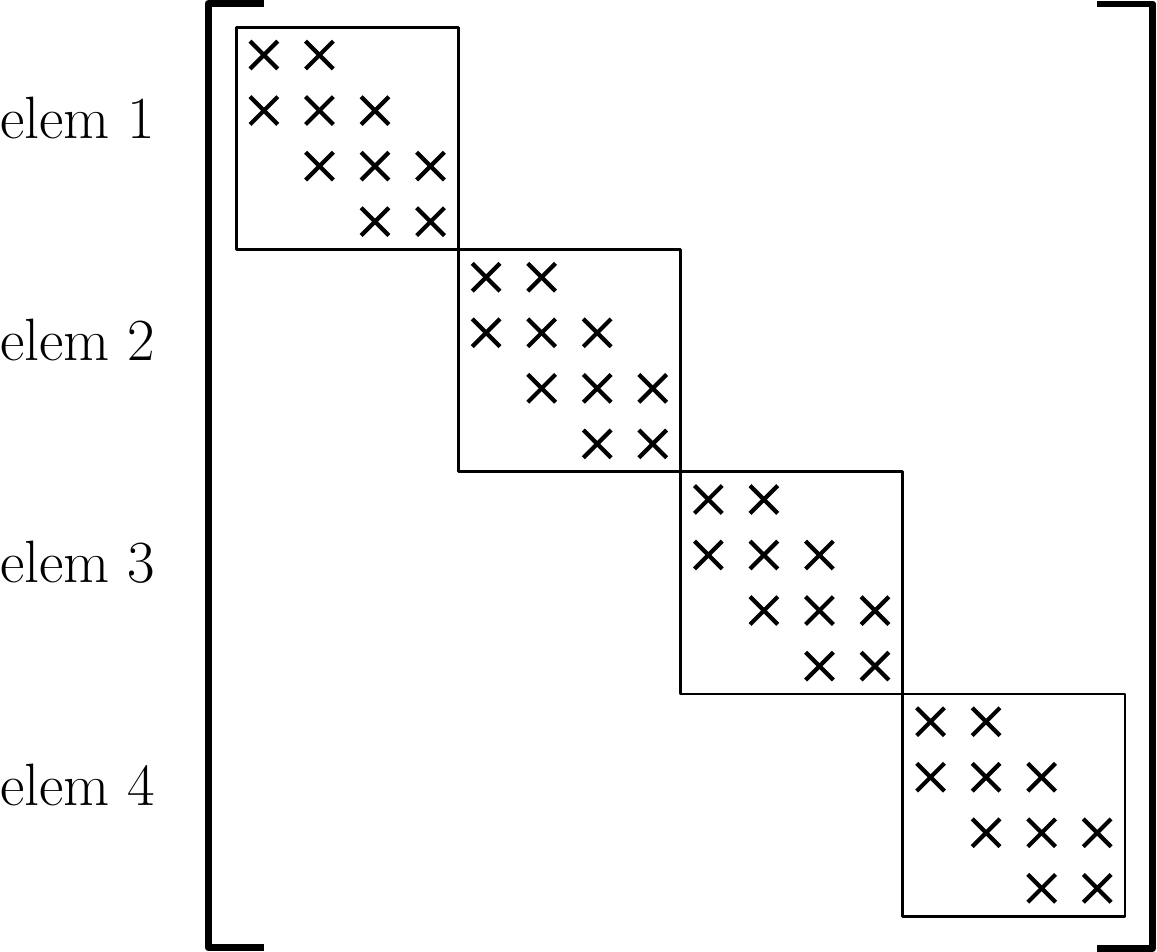}
\caption{Block diagonal structure of $J_n^z$ for a simplified case with 4 horizontal elements and 4 vertical layers.} \label{block_diag}
\end{center}
\end{figure}

This block diagonal structure gives several advantages. 
First, $\varphi_1(\Delta t J_n^z)$ 
can be written as 
\begin{equation}\varphi_1(\Delta t J_n^z) = \begin{bmatrix}
\varphi_1(\Delta t J_n^{z,1}) & 0 & \dots & 0\\
0 & \varphi_1(\Delta tJ_n^{z,2}) & \dots & 0\\
\vdots & \dots & \dots & \vdots\\
0 & 0 & \dots & \varphi_1(\Delta tJ_n^{z,N_h})
\end{bmatrix} \;. \label{parallel} \end{equation}
Hence, for every Voronoi cell $i$, the smaller matrices $\varphi_1(\Delta t J_n^{z,i})$ can be constructed one at the time, 
and the global matrix $\varphi_1(\Delta t J_n^z)$ is never assembled.
The evaluation of small matrices $\varphi_1(\Delta t J_n^{z,i})$ instead of a large one can significantly speed up the calculations, 
especially in a parallel setting. 
Expression \eqref{parallel} makes indeed the computation of $\varphi_1(\Delta t J_n^z)$ easy to implement in a parallel environment. 
Ideally, each matrix $\varphi_1(\Delta t J_n^{z,i})$ could be assigned to a different processor, if $N_x$ processors were available. 
This straightforward parallelization would be of great advantage to speed up the computational time, which is one of the major concerns in 
ocean modeling.

Expression \eqref{parallel} holds for any $\varphi_k$ with $k \ge 0$ and for any block diagonal matrix, as shown in the following proposition.
\begin{proposition}
Let $D$ be a block diagonal matrix $$D=\begin{bmatrix} D_1 & 0 & \dots & 0\\
0 & D_2 & \ddots & 0\\
\vdots & \ddots & \ddots & \vdots\\
0 & 0 & \dots & D_m \end{bmatrix}\;,$$
and let $\varphi_k$ be a $\varphi$-function of index $k$ defined as in \eqref{phi_def1} (or equivalently \eqref{phi_def2}). Then,
\begin{equation}
\label{eq:phi_blkdiag}
\varphi_k(D) = \begin{bmatrix}
\varphi_k(D_1) & 0 & \dots & 0\\
0 & \varphi_k(D_2) & \ddots & 0\\
\vdots & \ddots & \ddots & \vdots\\
0 & 0 & \dots & \varphi_k(D_m) \end{bmatrix} \;.
\end{equation}
\end{proposition}

\begin{proof} Using definition \eqref{phi_def2} and the block diagonal structure of $D$, we have 
\begin{align*}\varphi_k(D) & = \dfrac{1}{(k-1)!} \int_0^1 \exp((1-\sigma)D) \sigma^{k-1} d\sigma \\
& = \dfrac{1}{(k-1)!} \int_0^1 \begin{bmatrix}
\exp((1-\sigma)D_1) & 0 & \dots & 0\\
0 & \exp((1-\sigma)D_2) & \ddots & 0\\
\vdots & \ddots & \ddots & \vdots\\
0 & 0 & \dots & \;\exp((1-\sigma)D_m)\\
\end{bmatrix}
\sigma^{k-1} \de\sigma.
\end{align*}
Applying the integral to every entry of the matrix we have that the diagonal blocks become  
\begin{align*}
\frac{1}{(k-1)!} \int_0^1 \exp((1-\sigma) & D_i) \sigma^{k-1} \de\sigma = \varphi_k(D_i)\;,
\end{align*}
for all \(i = 1,\ldots,N_h\).
Thus, \eqref{eq:phi_blkdiag} has been verified.
\end{proof} 

Another benefit given by \eqref{parallel} relies in the different accuracy that can be chosen in the approximation of each matrix 
$\varphi_1(\Delta t J_n^{z,i})$. If a Krylov scheme is used, the dimension of the Krylov space chosen to approximate 
$\varphi_1(\Delta t J_n^{z,i})b_i$ can vary for every $i$, depending on the physics of the problem. 
A higher dimension (so more Krylov vectors) can be used for those horizontal elements that experience a great vertical 
mixing or transport, while a lower dimension can be adopted to approximate $\varphi_1(\Delta t J_n^{z,i})b_i$ for the remaining elements $i$. 
In most regions of the ocean, mixing is relatively small, and considering a lower subspace dimension for the cells that 
discretize those areas would considerably speed up the computational time without jeopardizing the accuracy of the approximation. 
Using a restarted Krylov method, the number of Krylov vectors can be significantly smaller than the dimension of $J_n^{z,i}$, 
i.e. $N_x$, this number can change at each restart.  
Similarly, if the scaling and squaring method described in section \ref{scaling} is adopted,
a different value of $M$ can be used for each matrix $\Delta t J_n^{z,i}$. For those Voronoi cells experiencing great 
vertical mixing or transport, the spectral radius of the corresponding $\Delta t J_n^{z,i}$ would be larger than for those cells 
impacted by relatively small mixing or transport. Hence, smaller values of $M$ are needed for the elements not experiencing 
processes occurring at fast time-scales, and since such elements are the great majority, significant speed ups can be obtained. 

\subsection{Computational Complexity} \label{com_com}

We complete this section with a brief discussion of the computational cost depending on the approach chosen 
for computing the $\varphi_k$-functions.  
For dense matrices, the scaling and squaring approach is \(\mathcal{O}(N^3_z)\) and probably expensive if $N_z$ is too large. 
\sara{However, the discrete matrices arising from one dimensional convection diffusion problems are banded, and tridiagonal for 
first and second order schemes, such as upwinding and central flux approximations. Therefore, the matrices resulting from the 
scaling and squaring approximation will also have a small bandwidth:}
for two banded matrices with bandwidth \(b\), the matrix product will be banded with bandwidth
\(2b\), and the product will cost \(\mathcal{O}((1+b)^2 N_z)\). 
\sara{In the context of the ETD method \eqref{stage1_pc_tracer}-\eqref{stage2_pc_tracer} we need to apply the same matrix function 
to two different right-hand sides. Moreover, if many tracers have to be computed, the computation of the matrix function has to be 
performed only once in each timestep. The cost of multiplying by the approximation of $\varphi_k(\Delta t J_n^{z,i})$ is 
$\mathcal{O}((1 + 2^M b)^2 N_z)$, and the overall cost for the scaling and squaring method and application to $N_{RHS}$ right-hand side 
is $\mathcal{O}(((1 + 2^M b) + N_{RHS})(1 + 2^M b)N_z)$. For the Krylov method from section \ref{restarted_kry}, we need to recompute the 
Arnoldi like decomposition for every right-hand side, which results in a minimum complexity of $\mathcal{O}(N_{RHS} (1+b) N_{Krylov} N_z)$, 
where $N_{Krylov}$ is the number of matrix-vector products required. Theoretic considerations and practical observations suggest that both 
$N_{Krylov}$ and $(1+2^M b)$ have to be chosen proportional to the vertical CFL number to obtain sufficiently stable and accurate results. 
Thus, for $N_{RHS}$ sufficiently large, the scaling and squaring method is competitive in theory, and yields improved results in practice 
due to the high efficiency of multiplying by the precomputed banded matrix.}

\sara{
The fact that we are dealing with tridiagonal matrices greatly reduces the overall computational cost. Moreover, we note that $N_z$ 
is typically of moderate size only. In \cite{ringler2013multi}, the authors choose $N_z = 40$, and 100 layers is a realistic value 
for present day ocean simulations. Even if this number will tend to grow in the future, it is fair to assume that its value will not 
exceed a thousand in the next decade. This also implies that the storage required for a single $\varphi_k$-function approximation is moderate.
}

\section{Numerical Tests} \label{tests}

In this section, numerical tests are presented to investigate the performances of our ETD solver and compare it with other ocean models.  
The ETD time-stepping scheme is implemented using the two approaches for the computation of the $\varphi$-functions described in section 
\ref{ETD}: the Krylov subspace scheme presented in \cite{eiermann2006restarted} and the scaling and squaring algorithm developed for $\varphi_k$ 
with $k \ge 0$. The performances given by the two implementations are investigated numerically in the case of one and multiple tracers. 
Comparisons with other semi-implicit schemes are also presented. 
Finally, \sara{the solver is applied in the context of a simple ocean model and compared to existing models}, 
to make sure that the proposed method is able to reproduce similar 
results under the same physical conditions. To do so, the whole primitive equation system is solved. 
All the tests in this section are 2D, namely one dimension in the horizontal and one in the vertical. 
They have been implemented in the in-house the C++ library FEMuS (\cite{FEMuSlibrary}), and for the Krylov subspace scheme, 
the SLEPc library has been used. 

\begin{remark}For all the tests, up-winding is used to discretize the horizontal and vertical advection. For the performance tests, 
first order up-winding is used for both the vertical and the horizontal, whereas a third-order up-winding scheme for the horizontal 
advection and a first-order scheme for the vertical are employed for the comparisons with other ocean models. 
In section \ref{tracer_eq_section}, the discretization of the tracer equation was presented using a central difference scheme, 
but in general it is always possible to move from a central difference scheme to up-winding, either first order or higher, by 
making an appropriate choice of the diffusion operators. 
Computationally, up-winding was necessary since central differences cause instabilities in advection-diffusion problems (\cite{manzini2008finite}). 
The advection schemes used by other ocean models are discussed in section \ref{primitive_eq_tests}.
\end{remark}

\subsection{Performance} \label{performance}
For the performance tests, the tracer equation is solved, considering $u$, $w$ and $h$ constant in time. 
The domain is a $10$ m $\times 10$ m box discretized with 12 elements in the horizontal and 100 layers in the vertical, hence 
$\Delta x = 0.83$ m and $\Delta z = 0.1$ m.
The velocity field is a circular, \sara{divergence-free field, which is tangential to the boundaries}. It is defined as
$$(u, w) = (-\psi_1(x)\psi_2'(z),\psi_1'(x)\psi_2(z))\;,$$
where
$$\psi_1(x) = 1 - \dfrac{(x - \frac{x_{max}}{2})^4}{(\frac{x_{max}}{2})^4}\;, \quad 
\psi_2(z) = 1 - \dfrac{(z - \frac{z_{min}}{2})^2}{(\frac{-z_{min}}{2})^2}\;,$$ 
with $x_{max} = 10$ and $z_{min} = -10$. 
Figure \ref{velocity} shows that, with this velocity field, the Voronoi cells close to the boundaries experience more vertical transport 
and mixing than those in the center.  
\begin{figure}[!b]
\begin{center}
\includegraphics[scale=0.45]{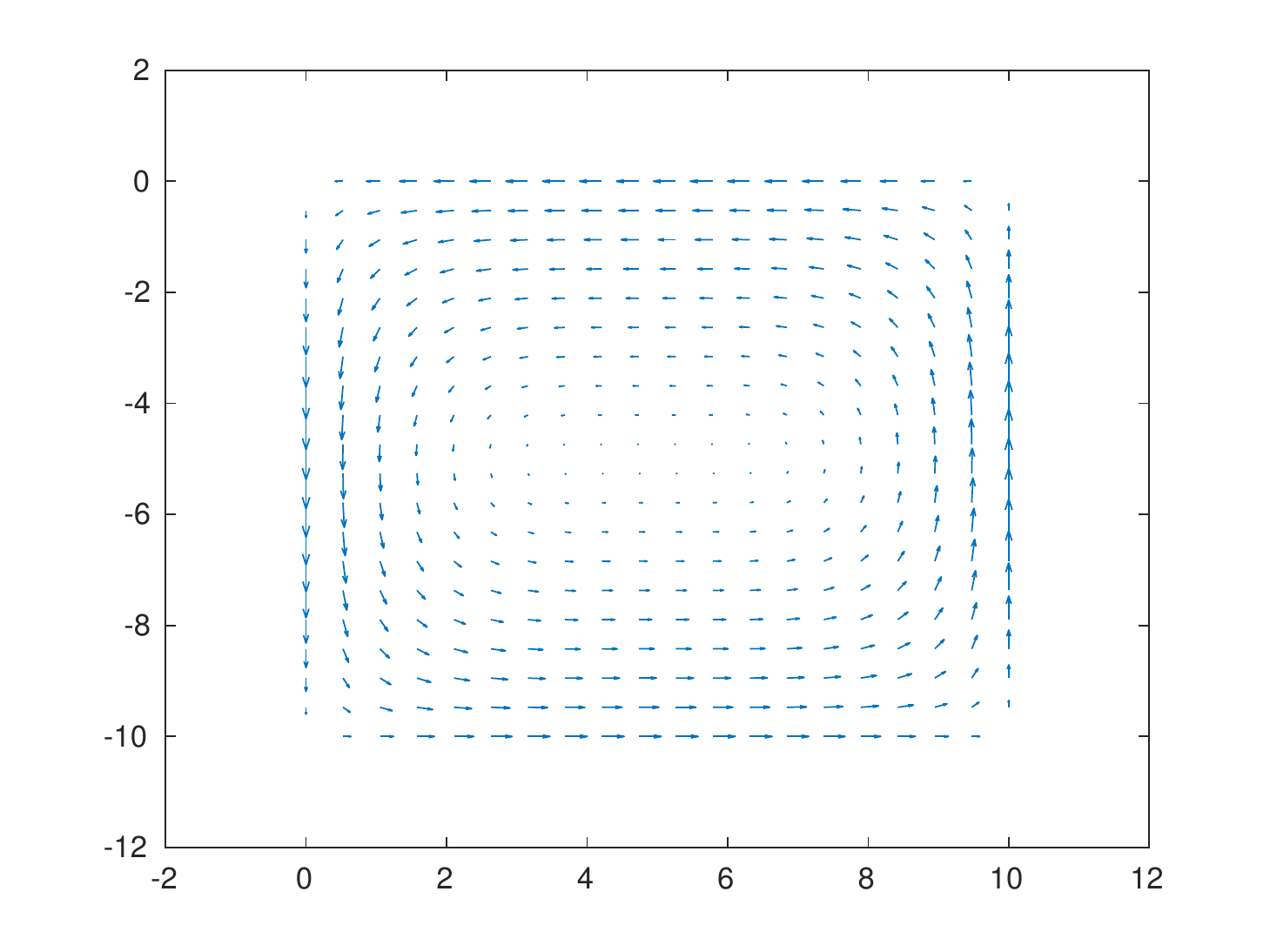}
\caption{Velocity field.} \label{velocity}
\end{center}
\end{figure}
The two CFL numbers, CFL$_x$ and CFL$_z$ are defined as
$$\mbox{CFL}_x = \dfrac{\max u \cdot dt}{\Delta x} \quad \mbox{and} \quad \mbox{CFL}_z = \dfrac{\max w \cdot dt}{\Delta z}\;.$$
With the given velocity field and discretization, the ratio $\dfrac{\mbox{CFL}_z}{\mbox{CFL}_x}$ is equal to 16.6. 
This implies that the transport and mixing in the vertical are significantly bigger than those in the horizontal. 
For the horizontal and vertical diffusion, the values chosen are $\kappa_{x} = 10^{-4}$ and $\kappa_{z} = 2.5 \cdot 10^{-5}$.

First, let us consider the case of a single tracer, and let us solve the equation with the second stage ETD method 
\eqref{stage1_pc_tracer}-\eqref{stage2_pc_tracer} and 
a semi-implicit method that uses implicit Euler to treat the vertical diffusion and RK4 for the rest of the terms. 
In the following, the latter scheme is denoted as RK4 + implicit Euler.
The initial condition given for the tracer, which we assume to be temperature, $\varTheta(x,z)$, is 
$$\varTheta(x,z)=\begin{cases}
5, & \text{$x<5$ m,} \\
30, & \text{$x\ge 5$ m.}
\end{cases}$$ 
This is the same initial condition given for the temperature in the lock exchange test case analyzed later.  
Table \ref{speed_up_over_RK} shows the times employed by the two time-stepping schemes to reach the steady state. 
For the ETD method, both the restarted Krylov and the scaling and squaring algorithm have been used for the evaluation of  
the $\varphi_1$ functions. 
When the restarted Krylov algorithm is used, the method is denoted as ETD2 Restarted Krylov, whereas ETD2 Scaling \& Squaring 
indicates the scheme using the scaling and squaring algorithm. 
With ETD2 Restarted Krylov, the vectors $\varphi(\Delta t J_n^{z,i})b_i$ are computed at each stage, without ever 
storing $\varphi_1(\Delta t J_n^{z,i})$. 
\sara{A naive implementation of ETD2 Scaling \& Squaring would consist in evaluating the matrices $\varphi_1(\Delta t J_n^{z,i})$ 
at each stage, but this multiple evaluation is actually unnecessary, since these matrices do not change from one stage to another. 
Therefore, as already mentioned in section \ref{com_com}, in our implementation of ETD2 Scaling \& Squaring the matrices are computed 
only at the first stage, stored, and then re-used at the second stage.}
Taking advantage of the physics of the problem, i.e. larger vertical transport and mixing for the Voronoi cells close to the boundaries,    
20 Krylov vectors have been used for the six external elements and 10 vectors for the central six with ETD2 Restarted Krylov. 
The highest possible dimension for the Krylov subspaces is 100, since the matrices $\Delta t J_n^{z,i}$ are $100 \times 100$.
When the scaling and squaring method is applied, $2^4$ is used as scaling factor for all the matrices $\Delta t J_n^{z,i}$.
As Table \ref{speed_up_over_RK} shows, with the ETD methods larger time steps can be taken than with RK4 + implicit Euler. The values of 
$\Delta t=3$ and $\Delta t=0.25$ were the largest values that could be taken for the two time-stepping schemes without compromising stability. 
The simulations end after 2000 time steps. 
The time employed by the ETD methods is, in all cases, smaller than the time required by RK4 + implicit Euler, because a time step 
12 times larger than that of RK4 + implicit Euler could be used. The computational time is reduced by $5.29$ times using 
ETD2 Restarted Krylov and $5.94$ using ETD2 Scaling \& Squaring.  
Another advantage of such methods over RK4 + implicit Euler is in the order of accuracy of the time-stepping scheme. The two-stage ETD method 
\eqref{stage1_pc_tracer}-\eqref{stage2_pc_tracer} is second order accurate, whereas RK4 + implicit Euler is only first order.

The results in Table \ref{speed_up_over_RK} are obtained computing the Jacobian $J_n$ at every time step, since in a realistic 
ocean model simulation, the values of $u$, $w$ and $h$ would change at every time step. 
Consequently, $J_n^z$ and $J_n^x$ would change.
Let us recall that $J_n^x$ is actually never built since in \eqref{stage2_pc_tracer}
$R_n$ is obtained for free, and $R_n^{1st\;stage}$ can be constructed by evaluating only the horizontal terms of 
$F(T_n^{1st\;stage})$.
\begin{remark}
If \(w\) does not change too much from one step to the other, we could actually fix $J_n^z$ at some instant of time $t_m$, and 
use $J_m^z$ in the following iterations, \sara{adding the appropriate error terms to the remainder}. 
This choice might still give an advantage to the scaling and squaring method, since we would  
having a constant in time linear operator starting from $t_m$, without compromising accuracy. 
In section \ref{sec:conclusions}, we briefly discuss how this could be done as a future work. 
\end{remark}\label{new_option}

%
%
\begin {table}[!t] 
\setlength\tabcolsep{5.5pt} 
\begin{center}
	\begin{tabular}{|c|c|c|c|c|} \hline	
	\multicolumn{4}{|c|}{1 tracer} \\ \cline {1-4}
	\multicolumn{1}{|c|} {scheme} 
	& \multicolumn{1}{|c|} {dt} & \multicolumn{1}{|c|} {time steps} & \multicolumn{1}{|c|} {computational time}\\ \hline
	     ETD2 Restarted Krylov & 3 & 2000 & 112.7933 \\ \hline 
	     ETD2 Scaling \& Squaring & 3 & 2000 & 100.3743 \\ \hline 
	     RK4 + implicit Euler & 0.25 & 24000 & 596.6595 \\ \hline
	\end{tabular}
\end{center}
\caption{Computational times considering one tracer. All times are in seconds (s).}
\label{speed_up_over_RK}
\end{table}
%

%
\sara{
\begin {table}[!t] 
\setlength\tabcolsep{5.5pt} 
\begin{center}
	\begin{tabular}{|c|c|c|c|c|} \hline	
	\multicolumn{5}{|c|}{Multiple tracers} \\ \cline {1-5}
	\multicolumn{1}{|c|} {} & \multicolumn{2}{|c|} {ETD2 Scaling \& Squaring} & \multicolumn{2}{|c|} {ETD2 Restarted Krylov} \\ \cline {1-5}
	\multicolumn{1}{|c|} {n. of tracers} & \multicolumn{1}{|c|} {time} & \multicolumn{1}{|c|} {time for each tracer} 
	& \multicolumn{1}{|c|} {time} & \multicolumn{1}{|c|} {time for each tracer} \\ \hline
	     1 tracers & 100.37 & 100.37 & 112.79 & 112.79 \\ \hline
	     2 tracers & 109.10 & 54.55 & 202.40 & 101.2 \\ \hline 
	     4 tracers & 133.58 & 33.40 & 383.26 & 95.82 \\ \hline 
	     6 tracers & 153.32 & 25.55 & 562.67 & 93.78 \\ \hline
	\end{tabular}
\end{center}
\caption{Computational times (in seconds, s) for one, two, four and six tracers, using ETD2 Scaling \& Squaring 
where the $\varphi_1$ functions are evaluated only for one tracer and ETD2 Restarted Krylov. 
The time step used is $\Delta t =3$ and the total number of time steps is 2000.}
\label{multiple_tracers}
\end{table}
}

Now, let us consider the case of multiple tracers. Having multiple tracers implies that multiple equations of the form \eqref{tracer_hor_discr}
need to be solved. The same test case as for the single tracer simulations is used. For simplicity, the same initial condition 
is given but with different numerical values for each tracer, so that the steady state changes for 
each tracer. In the upcoming tests, the time-stepping scheme RK4 + implicit Euler is not considered. 
Having multiple tracers may seem a straightforward task to deal with, but in ocean modeling the amount of tracers can be quite large, 
ranging from 1 to 70. The computational time employed to solve these equations can then drag down the run time of the whole simulation. 
For this reason, it is not ideal to just naively solve multiple tracer equations. 
When dealing with multiple equations of the form \eqref{tracer_hor_discr}, one can take advantage of the fact that 
the Jacobian $J_n$ is the same for all of them, no matter how many they are. Therefore, the matrices $J_n^{z,i}$ and 
$\varphi_1(\Delta t J_n^{z,i})$ can be computed only for one tracer at every time step, and then used to solve all the tracer equations. 
What changes for each equation is just the right hand side $b_i$ that multiples $\varphi_1(\Delta t J_n^{z,i})$ for a given $i$. 
Combining this strategy with ETD2 Scaling \& Squaring, we simply compute the matrices $\varphi_1(\Delta t J_n^{z,i})$ at the 
first stage for only one tracer. 
Table \ref{multiple_tracers} shows the computational times obtained with this approach in the case of 1, 2, 4 and 6 tracers 
\sara{and compares it with the performances obtained with ETD2 Restarted Krylov}. 
From the table, we see that an important advantage of \sara{computing the $\varphi_1$-functions only for one tracer} 
lies in the increased time saving when more tracers are added. 
\sara{The numbers reported in column three scale rapidly, meaning that, every time that a tracer is added, just a small amount of cost 
is dedicated for the solution of the new equation. 
The majority of the time is devoted to the computation of the $\varphi_1$ functions, whereas the evaluation of the products 
$\varphi_1(\Delta t J_n^{z,i}) b_i$ is cheap. For example, the computational time moving from one to six tracers increases only by 
a factor of 0.5, hence an even bigger advantage is expected when 40 or 50 tracers are present.} 
Using ETD2 Restarted Krylov, the matrices $\varphi_1(\Delta t J_n^{z,i})$ are re-computed 
for every tracer and this causes the times to roughly double when the number of tracers is doubled. 
\sara{Thus, every time a tracer is added to the system, the computational cost significantly increases.} 
The speed-up over the Restarted Krylov implementation is almost 50\% with 2 tracers, and it keeps increasing up to 72.75\% for 6 tracers. 
Therefore, a significant amount of time is saved by taking advantage of having the same linear part for all tracers.  

\begin{remark}For this test, we consciously considered a small number of elements in the horizontal because one of the biggest 
advantage of the ETD method \eqref{stage1_pc_tracer}-\eqref{stage2_pc_tracer} is to be heavily parallel, and so each 
process will have only a few numbers of Voronoi cells to deal with. 
Thus, using a fine grid in the horizontal is not relevant for this performance comparison.\end{remark}

\subsection{Comparison with other ocean models} \label{primitive_eq_tests}
We now compare the proposed ETD method to other ocean models under identical conditions (whenever the data is available). 
Two benchmark tests from \cite{ilicak2012spurious} are performed, addressing the solution of the primitive equations. 
These tests show, for different initial conditions, the temperature distribution at a given instant of time. 
Since hydrostaticity is assumed in this work, the primitive equations are described by the incompressible Boussinesq 
equations in hydrostatic balance. For the tests, the tracer equation is coupled to the dynamics.  
A second-order ETD solver is employed for the dynamics system, specifically Exponential Rosenbrock Euler. 
The tests presented are two-dimensional in $(x,z)$, namely one dimension in the horizontal and one dimension in the vertical, and 
a linear equation of state is used. The equation of state has the form
$$\rho = \rho_{ref} - \alpha(\varTheta - \varTheta_{ref})\;,$$
where $\rho_{ref} = 1000$ kg m$^3$, $\alpha = 0.2$ kg m$^3$ C$^{-1}$ and $\varTheta_{ref} = 5^{\circ}$ C, 
so that density depends only on temperature.
For both tests, our solutions are compared with those obtained by MPAS-Ocean, and for the second test, 
a comparison is made also with MITgcm and MOM. 
The advection schemes used by these three codes and our ETD scheme are different. 
The tracer advection scheme we use is a third-order upwinding scheme for the horizontal advection and a first-order 
upwinding scheme for the vertical, whereas MPAS-Ocean computes high- and low-order estimate of the tracer flux 
which are then blended using the flux-corrected transport scheme of Zalesak~\cite{zalesak1979fully}. 
MITgcm uses a 7th-order monotonicity preserving advection scheme~\cite{marshall1997finite}, wheres MOM employs a third-order accurate
scheme based on a multi-dimensionsal piecewise parabolic method~\cite{griffies2009elements}. 
We remark that another different between the four solvers relies in the treatment of the vertical advection, which is explicit 
for the three ocean models and implicit in our method. 
The results shown in this section are obtained using the restarted Krylov subspace method for computing the $\varphi_1$ functions. 
Solutions obtained with the scaling and squaring method are nearly identical.
The comparisons made are only qualitative, to show that our tracer solver is able to reproduce results comparable to those 
of other models. 
  
\subsubsection{Lock Exchange Test Case}
The lock exchange test case may be thought of as two basins of water with different temperatures that start interacting at time zero.
The domain is a $64,000 \times 20$ rectangle, where $64,000$ is the $x$ size and $20$ is the $z$ size. All dimensions are in meters (m). 
The cell sizes are $\Delta x = 500$ and $\Delta z = 1$, i.e.\ 128 elements are considered
in the horizontal, while twenty layers are considered in the vertical. 
The initial condition for temperature is 
$$\varTheta(x,z)=\begin{cases}
5, & \text{$x<32,000$ m,} \\
30, & \text{$x\ge 32,000$ m,}
\end{cases}$$
so, warm water flows over the cold water from right to left, and viceversa once the domain boundaries are touched. 
The initial condition for velocity is $u=0$ in every layer.
The values for the horizontal and vertical viscosity are $100$ m$^2$ s$^{-1}$ and $0.0001$ m$^2$ s$^{-1}$, respectively, 
while all tracer diffusions are turned off. 
The simulation stops at 17 h, and the same $dt$ used by MPAS-Ocean in \cite{petersen2015evaluation} is adopted,
i.e. $dt=60$ s. 
Figure~\ref{lock} shows the temperature distribution at 17 h obtained with our ETD solver and with MPAS-Ocean.
\begin{figure}
\begin{center}
(a) \includegraphics[height=3.0cm,width=6.4cm]{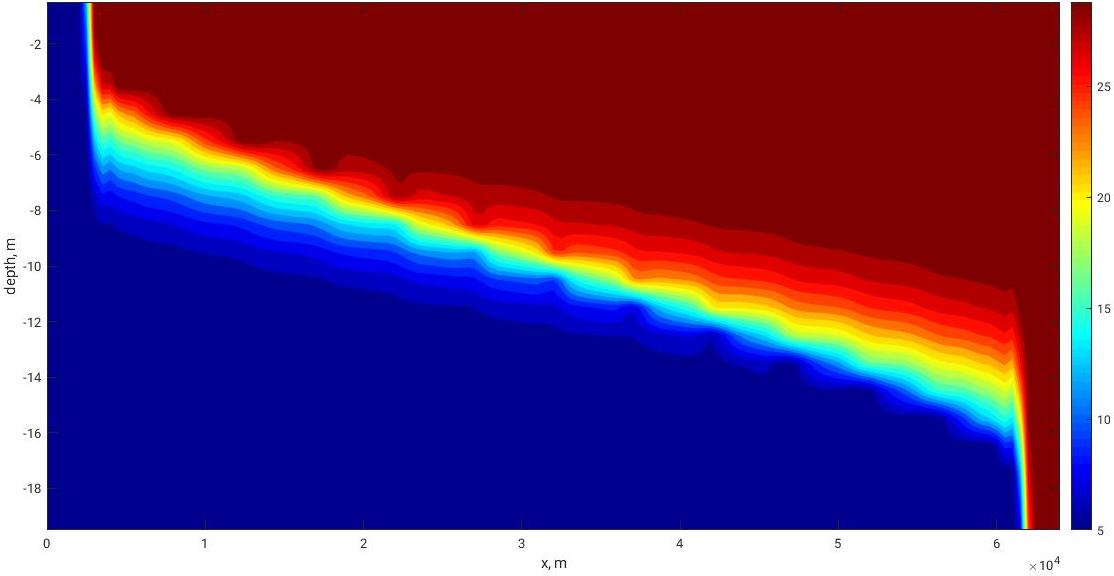} \\
(b) \includegraphics[scale=0.45]{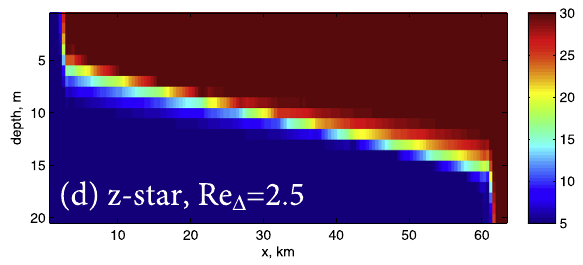}
\caption{Temperature distribution for the lock exchange test case with $\nu_h=100$: (a) proposed ETD solver, 
(b) MPAS-Ocean (\cite{petersen2015evaluation}).} \label{lock}
\end{center}
\end{figure}
The simulation performed by MPAS-Ocean uses z-star vertical coordinates, while we use z-level. This difference does not compromises the 
comparison between the two distributions, since, as reported in \cite{petersen2015evaluation}, results for z-level and z-star vertical coordinate settings in MPAS-Ocean are nearly identical. 
As Figure \ref{lock} shows, the two distributions are comparable.  
With z-type coordinates, the intermediate layers are expected to have temperatures in between $5^{\circ}$ C and $30^{\circ}$ C, and this 
behavior is visible in both Figure \ref{lock} a) and Figure \ref{lock} b). The right front location at 17 h is very similar: with MPAS-Ocean 
the front is at 62 km, while with our ETD solver is at 62.4 km. The location obtained with the proposed ETD solver coincides with the 
theoretical prediction for this test based on the speed of a gravity current in a rectangular channel \cite{benjamin1968gravity}. 
A difference between the two temperature distributions is in the amplitude of the mixing. 
The interface between the density layers is sharper in Figure \ref{lock} b), and this is probably due to the different advection scheme 
used by MPAS-Ocean and to the exponential treatment of the vertical advection in our ETD solver. 

\subsubsection{Internal Waves Test Case}
\begin{figure}[!t]
\begin{center}
(a) \includegraphics[scale=0.3]{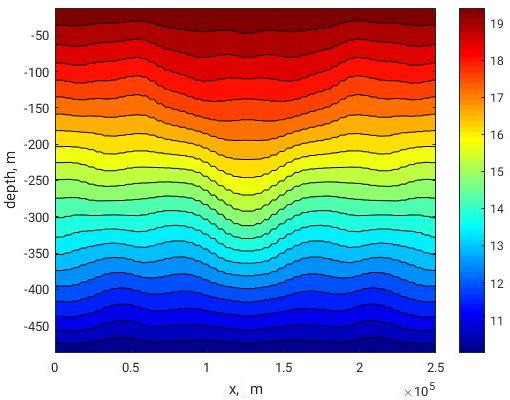}
(b) \includegraphics[height=4.2cm,width=5.0cm]{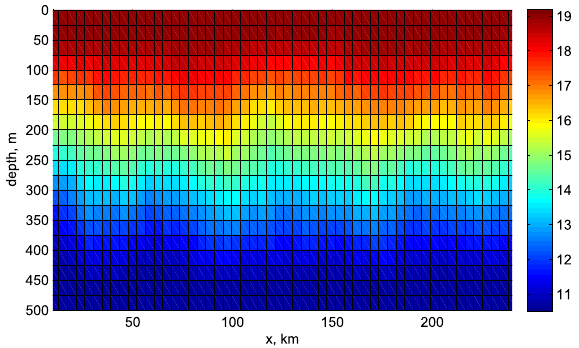}\\
(c) \includegraphics[scale=0.53]{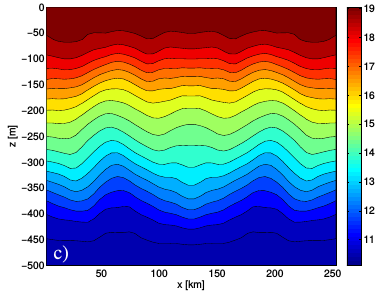}
(d) \includegraphics[scale=0.53]{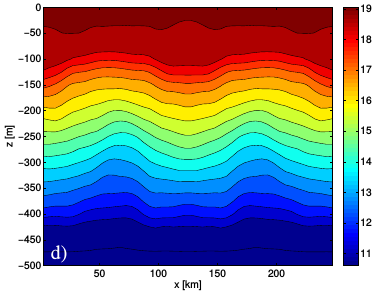}
\caption{Temperature distribution for the internal waves test case: (a) our ETD solver, (b) MPAS-Ocean (\cite{petersen2015evaluation}),
(c) MITgcm (\cite{ilicak2012spurious}), (d) MOM (\cite{ilicak2012spurious}).} \label{internal}
\end{center}
\end{figure} 
Internal waves are waves that oscillate within the interior of the ocean, rather than on its surface. 
They generate when the interface between layers of different water densities is disturbed. 
This test was chosen for further validation since linear internal waves tend to produce vertical mixing
in ocean models, especially when z-level and z-star coordinates are employed \cite{gouillon2010internal}. 
The domain is a $250,000 \times 500$ rectangle, where the dimensions are given in meters. 
The cell sizes are $\Delta x = 5000$ and $\Delta z = 25$, i.e. 50 elements are considered in the horizontal, while 
twenty layers are considered in the vertical. 
The initial temperature distribution is $\varTheta_0(z) + \varTheta'(x,z)$, with
\begin{align}\varTheta_0(z) & = \varTheta_{bot}+(\varTheta_{top}-\varTheta_{bot})\dfrac{z_{bot}-z}{z_{bot}}\,,\;\;\mbox{and} \\
\varTheta'(x,z) & = -A \cos \Big(\frac{\pi}{2L}(x-x_0)\Big) \sin\Big(\pi\frac{z + 0.5 \Delta z}{z_{bot} + 0.5 \Delta z}\Big)\,,\end{align}
where $\varTheta_{bot} = 10.1^{\circ}$ C, $\varTheta_{top} = 20.1^{\circ}$ C, $z_{bot} = -487.5$ m,
$L = 50$ km, $x_0 = 125$ km, $x_0 - L < x < x_0 + L$, $\Delta z = 25$ m, and $A = 2^{\circ}$ C. 
This means that we initially have a small temperature perturbation in each layer that induces wave propagation out from the center. 
This behavior is similar to that found in realistic global simulations. 
The initial condition for velocity is $u=0$ in every layer.
The values for the horizontal and vertical viscosity are $0.01$ m$^2$ s$^{-1}$ and $0.0001$ m$^2$ s$^{-1}$, respectively, 
while all tracer diffusions are, again, turned off. 
Unlike the previous test, this simulation may proceed indefinitely, but we choose to stop it at 200 days. 
The time step used is $dt=300$ s, which, again, is the same adopted by MPAS-Ocean.
Figure \ref{internal} shows the temperature distribution at 200 days obtained with our ETD solver and three other models: 
MPAS-Ocean, MITgcm and MOM. 
In general, the four temperature distributions are comparable.
The mixing is visible in all models, despite is different among the four solutions.  
Again, the four codes present four different advection schemes for the tracer equation, and this could explain the discrepancies in the 
solutions.  

\section{Concluding Remarks}
So far, in the analysis and in all tests performed, we assumed a zero forcing term for the tracer equations. 
Now, the more general case of a non-zero forcing term is considered, and we show that the analysis made above is still valid. 

Forcing terms introduce the contributions given by external factors like penetrative solar radiation and surface boundary conditions  \cite{madec2015nemo}. By treating these terms explicitly, i.e. excluding them from the matrix exponential, we can still follow the 
operator splitting procedure shown in section \ref{jac_splitting}. 
Let us assume that the forcing term $\mathcal{F^T}$ in \eqref{tracer} is non-zero. 
Let us define the Jacobian $J$ as the Jacobian of the associate equation with a zero forcing term. 
In this way, the matrices $J_n^z$ and $J_n^x$ are the same as for the zero forcing term case and the reminder is now 
given by $\mathcal{F^T}$ plus $J_n^x T$. 
Hence, no matter the nature of $\mathcal{F^T}$ and no matter its form, we can still apply the method 
\eqref{stage1_pc_tracer}-\eqref{stage2_pc_tracer} with the same $J_n^z$ matrix as for the zero forcing term case. 
In this way, when multiple tracers are present, they will all have the same linear part, and so the matrices $\varphi_1(J_n^{z,i})$ 
just need to be computed for one tracer and re-used for the others, as we did in the numerical tests. 
The possibility of re-using the matrices $\varphi_1(J_n^{z,i})$ even in the more general case of a non-zero forcing term  
is a great feature of the method that gives a saving in computational time that increases with the number of tracers.
What changes in the scheme \eqref{stage1_pc_tracer}-\eqref{stage2_pc_tracer} is the difference $R_n^{1st\;stage} - R_n$ 
that appears in the second stage, since in both reminders there is now the contribution of $\mathcal{F^T}$.

A possible issue that may occur with the introduction of forcing terms relies in the associated time-scales.  
By adding a forcing term, new physical or chemical processes are taken into account, in particular biochemical reaction terms 
may be included. If the time-scales associated with such processes are comparable with those associate with the 
vertical transport and mixing, then these terms must be included in the matrix exponential to correctly account for them. 
This scenario may occur for some tracers, but for many passive tracers the time-scales associated with their dissolution in the ocean 
are relatively long \cite{siberlin2011oceanic}. The dissolution of gases, for example, could take a long time to reach equilibrium, 
from decades to centuries. 

\section{Conclusions}
\label{sec:conclusions}

In this work, we developed an ETD solver for the tracer equations appearing in ocean modeling. The linear operator has been split 
in a vertical and horizontal part, the former was treated with a matrix exponential, whereas the latter was handled explicitly.
The need to treat the vertical terms exponentially is due to the fast time scales that govern these terms for instance in the case of eddies 
or when two bodies of water with different temperatures meet.
The ETD scheme was implemented using two methods to compute the $\varphi_k$-functions, i.e. the Krylov subspace method presented in
\cite{eiermann2006restarted}, and a scaling and squaring method that we developed to evaluate $\varphi_k$-functions with $k \ge 0$.
The proposed time-stepping scheme has been compared with the semi-implicit scheme RK4 + implicit Euler and a significant 
speed-up were observed for both implementations, due to much bigger time-step sizes that could be taken with the ETD method.
The case of multiple tracer equations was also addressed. Exploiting the fact that all these equations have the same linear part, the matrices $\varphi_1(J_n^{z,i})$ could be computed for one tracer and then re-used for the others.
This approach resulted in a significant advantage in terms of computational time, even up to a $82.04$\% gain over an 
implementation where the matrices were evaluated for every tracer.
Finally, we coupled the tracer equations with the dynamics system to make comparisons with other ocean models.
Two benchmark tests were performed and both showed that the results obtained with the proposed ETD scheme were comparable with the ones 
obtained with existing ocean models.

Future work will be on developing a local ETD time-stepping scheme, where different time-steps are used on different sub-domains, 
depending on their associated time-scales. As discussed in remark \ref{new_option}, another interesting \sara{extension} would be investigating 
another choice of the linear operator used in \eqref{stage1_pc_tracer}-\eqref{stage2_pc_tracer}, i.e. instead of having a time dependent \(J_n^z\), 
a fix \(J^z_{m}\) may be used starting at some instant of time $t_m$. 
A scaling and squaring method would particularly benefit from this choice, since in this context the assembly cost of matrices is not an issue 
and we can directly reuse the old propagator \(\varphi_1(\Delta{t} J^z_{m}\)).  
Then, the term \(J_n^z - J^z_{m}\) is part of the remainder, which would be not stiff as long as \(\Delta{t} < \abs{w(t_m) - w(t_n)} / \Delta{z}\). 

\section*{Acknowledgments} 
This work was supported by the US Department of Energy Office of Science under grants DE-SC0016591 and DE-SC020418, 
by the Fog Research Institute under contract no.~FRI-454, and in part, by UT-Battelle, LLC, 
under Contract No. DE-AC05-00OR22725 with the U.S. Department of Energy. Accordingly, the U.S. Government retains 
a non-exclusive, royalty-free license to publish or reproduce the published form of this contribution, or allow 
others to do so, for U.S. Government purposes.

\bibliography{references}

\begin{thebibliography}{10}
\expandafter\ifx\csname url\endcsname\relax
  \def\url#1{\texttt{#1}}\fi
\expandafter\ifx\csname urlprefix\endcsname\relax\def\urlprefix{URL }\fi
\expandafter\ifx\csname href\endcsname\relax
  \def\href#1#2{#2} \def\path#1{#1}\fi

\bibitem{siberlin2011oceanic}
C.~Siberlin, C.~Wunsch, Oceanic tracer and proxy time scales revisited.

\bibitem{chiodaroli2017existence}
E.~Chiodaroli, M.~Mich{\'a}lek, Existence and non-uniqueness of global weak
  solutions to inviscid primitive and {B}oussinesq equations, Communications in
  Mathematical Physics 353~(3) (2017) 1201--1216.

\bibitem{oliger1978theoretical}
J.~Oliger, A.~Sundstr{\"o}m, Theoretical and practical aspects of some initial
  boundary value problems in fluid dynamics, SIAM Journal on Applied
  Mathematics 35~(3) (1978) 419--446.

\bibitem{ringler2013multi}
T.~Ringler, M.~Petersen, R.~L. Higdon, D.~Jacobsen, P.~W. Jones, M.~Maltrud, A
  multi-resolution approach to global ocean modeling, Ocean Modelling 69 (2013)
  211--232.

\bibitem{smith2010parallel}
R.~Smith, P.~Jones, B.~Briegleb, F.~Bryan, G.~Danabasoglu, J.~Dennis,
  J.~Dukowicz, C.~Eden, B.~Fox-Kemper, P.~Gent, et~al., The parallel ocean
  program ({POP}) reference manual ocean component of the community climate
  system model ({CCSM}) and community earth system model ({CESM}), Rep.
  LAUR-01853 141 (2010) 1--140.

\bibitem{dukowicz1994implicit}
J.~K. Dukowicz, R.~D. Smith, Implicit free-surface method for the
  {B}ryan-{C}ox-{S}emtner ocean model, Journal of Geophysical Research: Oceans
  99~(C4) (1994) 7991--8014.

\bibitem{higdon2005two}
R.~L. Higdon, A two-level time-stepping method for layered ocean circulation
  models: further development and testing, Journal of Computational Physics
  206~(2) (2005) 463--504.

\bibitem{pieper2019exponential}
K.~Pieper, K.~C. Sockwell, M.~Gunzburger, Exponential time differencing for
  mimetic multilayer ocean models, arXiv preprint arXiv:1901.08116.

\bibitem{hochbruck2010exponential}
M.~Hochbruck, A.~Ostermann, Exponential integrators, Acta Numerica 19 (2010)
  209--286.

\bibitem{eiermann2006restarted}
M.~Eiermann, O.~G. Ernst, A restarted {K}rylov subspace method for the
  evaluation of matrix functions, SIAM Journal on Numerical Analysis 44~(6)
  (2006) 2481--2504.

\bibitem{saad1992analysis}
Y.~Saad, Analysis of some {K}rylov subspace approximations to the matrix
  exponential operator, SIAM Journal on Numerical Analysis 29~(1) (1992)
  209--228.

\bibitem{al2009new}
A.~H. Al-Mohy, N.~J. Higham, A new scaling and squaring algorithm for the
  matrix exponential, SIAM Journal on Matrix Analysis and Applications 31~(3)
  (2009) 970--989.

\bibitem{dieci2000pade}
L.~Dieci, A.~Papini, Pad{\'e} approximation for the exponential of a block
  triangular matrix, Linear Algebra and its Applications 308~(1-3) (2000)
  183--202.

\bibitem{higham2005scaling}
N.~J. Higham, The scaling and squaring method for the matrix exponential
  revisited, SIAM Journal on Matrix Analysis and Applications 26~(4) (2005)
  1179--1193.

\bibitem{moler2003nineteen}
C.~Moler, C.~Van~Loan, Nineteen dubious ways to compute the exponential of a
  matrix, twenty-five years later, SIAM review 45~(1) (2003) 3--49.

\bibitem{najfeld1995derivatives}
I.~Najfeld, T.~F. Havel, Derivatives of the matrix exponential and their
  computation, Advances in applied mathematics 16~(3) (1995) 321--375.

\bibitem{archibald2011multiwavelet}
R.~Archibald, K.~J. Evans, J.~Drake, J.~B. White~III, Multiwavelet
  discontinuous {G}alerkin-accelerated {E}xact {L}inear {P}art ({ELP}) method
  for the shallow-water equations on the cubed sphere, Monthly Weather Review
  139~(2) (2011) 457--473.

\bibitem{ilicak2012spurious}
M.~Il{\i}cak, A.~J. Adcroft, S.~M. Griffies, R.~W. Hallberg, Spurious
  dianeutral mixing and the role of momentum closure, Ocean Modelling 45 (2012)
  37--58.

\bibitem{petersen2015evaluation}
M.~R. Petersen, D.~W. Jacobsen, T.~D. Ringler, M.~W. Hecht, M.~E. Maltrud,
  Evaluation of the arbitrary {L}agrangian--{E}ulerian vertical coordinate
  method in the {MPAS}-{O}cean model, Ocean Modelling 86 (2015) 93--113.

\bibitem{ringler2011exploring}
T.~D. Ringler, D.~Jacobsen, M.~Gunzburger, L.~Ju, M.~Duda, W.~Skamarock,
  Exploring a multiresolution modeling approach within the shallow-water
  equations, Monthly Weather Review 139~(11) (2011) 3348--3368.

\bibitem{clancy2013use}
C.~Clancy, J.~A. Pudykiewicz, On the use of exponential time integration
  methods in atmospheric models, Tellus A: Dynamic Meteorology and Oceanography
  65~(1) (2013) 20898.

\bibitem{gaudreault2016efficient}
S.~Gaudreault, J.~A. Pudykiewicz, An efficient exponential time integration
  method for the numerical solution of the shallow water equations on the
  sphere, Journal of Computational Physics 322 (2016) 827--848.

\bibitem{luan2019further}
V.~T. Luan, J.~A. Pudykiewicz, D.~R. Reynolds, Further development of efficient
  and accurate time integration schemes for meteorological models, Journal of
  Computational Physics 376 (2019) 817--837.

\bibitem{hochbruck2005explicit}
M.~Hochbruck, A.~Ostermann, Explicit exponential {R}unge--{K}utta methods for
  semilinear parabolic problems, SIAM Journal on Numerical Analysis 43~(3)
  (2005) 1069--1090.

\bibitem{moler1978nineteen}
C.~Moler, C.~Van~Loan, Nineteen dubious ways to compute the exponential of a
  matrix, twenty-five years later, SIAM review 20~(1) (1978) 801--836.

\bibitem{sidje1998expokit}
R.~B. Sidje, Expokit: a software package for computing matrix exponentials, ACM
  Transactions on Mathematical Software (TOMS) 24~(1) (1998) 130--156.

\bibitem{hochbruck1997krylov}
M.~Hochbruck, C.~Lubich, On {K}rylov subspace approximations to the matrix
  exponential operator, SIAM Journal on Numerical Analysis 34~(5) (1997)
  1911--1925.

\bibitem{roman2015slepc}
J.~E. Roman, C.~Campos, E.~Romero, A.~Tom{\'a}s, Slepc users manual, D.
  Sistemes Informatics i Computaci{\'o}, Universitat Politecnica de Valencia,
  Tech. Rep. DSIC-II/24/02-Revision 3.

\bibitem{afanasjew2008implementation}
M.~Afanasjew, M.~Eiermann, O.~G. Ernst, S.~G{\"u}ttel, Implementation of a
  restarted {K}rylov subspace method for the evaluation of matrix functions,
  Linear Algebra and its applications 429~(10) (2008) 2293--2314.

\bibitem{beylkin1998new}
G.~Beylkin, J.~M. Keiser, L.~Vozovoi, A new class of time discretization
  schemes for the solution of nonlinear {PDE}s, Journal of computational
  physics 147~(2) (1998) 362--387.

\bibitem{koikari2007error}
S.~Koikari, An error analysis of the modified scaling and squaring method,
  Computers \& Mathematics with Applications 53~(8) (2007) 1293--1305.

\bibitem{SkaflestadWright:2009}
B.~Skaflestad, W.~M. Wright, The scaling and modified squaring method for
  matrix functions related to the exponential, Applied Numerical Mathematics
  59~(3-4) (2009) 783--799.

\bibitem{ketcheson2013optimal}
D.~Ketcheson, A.~Ahmadia, Optimal stability polynomials for numerical
  integration of initial value problems, Communications in Applied Mathematics
  and Computational Science 7~(2) (2013) 247--271.

\bibitem{FEMuSlibrary}
E.~Aulisa, S.~Bna, G.~Bornia, \href{https://github.com/FeMTTU/femus}{Femus
  {L}ibrary}.
\newline\urlprefix\url{https://github.com/FeMTTU/femus}

\bibitem{manzini2008finite}
G.~Manzini, A.~Russo, A finite volume method for advection--diffusion problems
  in convection-dominated regimes, Computer Methods in Applied Mechanics and
  Engineering 197~(13-16) (2008) 1242--1261.

\bibitem{zalesak1979fully}
S.~T. Zalesak, Fully multidimensional flux-corrected transport algorithms for
  fluids, Journal of Computational Physics 31~(3) (1979) 335--362.

\bibitem{marshall1997finite}
J.~Marshall, A.~Adcroft, C.~Hill, L.~Perelman, C.~Heisey, A finite-volume,
  incompressible {N}avier {S}tokes model for studies of the ocean on parallel
  computers, Journal of Geophysical Research: Oceans 102~(C3) (1997)
  5753--5766.

\bibitem{griffies2009elements}
S.~Griffies, M.~Schmidt, M.~Herzfeld, Elements of {MOM}4p1, {GFDL} {O}cean
  {G}roup {T}echnical {R}eport 6, NOAA/Geophysical Fluid Dynamics Laboratory.

\bibitem{benjamin1968gravity}
T.~B. Benjamin, Gravity currents and related phenomena, Journal of Fluid
  Mechanics 31~(2) (1968) 209--248.

\bibitem{gouillon2010internal}
F.~Gouillon, Internal wave propagation and numerically induced diapycnal mixing
  in oceanic general circulation models, (Ph.D. thesis) Florida State
  University.

\bibitem{madec2015nemo}
G.~Madec, et~al., {NEMO} ocean engine.

\end{thebibliography}

\end{document}